\documentclass[12pt, reqno]{amsart}
\usepackage{amsmath, amstext, amsbsy, amssymb, amscd}
\usepackage{amsmath}
\usepackage{amsxtra}
\usepackage{amscd}
\usepackage{amsthm}
\usepackage{amsfonts}
\usepackage{amssymb}
\usepackage{eucal}
\usepackage{color}
\usepackage[all]{xy}

\newtheorem{theorem}{Theorem}[section]
\newtheorem{lemma}[theorem]{Lemma}
\newtheorem{prop}[theorem]{Proposition}
\newtheorem{corollary}[theorem]{Corollary}
\theoremstyle{definition}
\newtheorem{conj}[theorem]{Conjecture}
\newtheorem{defn}[theorem]{Definition}
\newtheorem{example}[theorem]{Example}
\newtheorem{remark}[theorem]{Remark}

\newtheorem{question}[theorem]{Question}
\newtheorem{conventions}[theorem]{Conventions}

\numberwithin{equation}{section}

\def\gen{\text{gen}}

\def\Aut{\text{Aut}}

\def\Stab{\text{Stab}}

\def\ad{\text{ad}}

\def\pr{\text{pr}}
\def\Der{\text{Der}}

\def\GL{{\text{GL}}}

\def\Lie{\text{Lie}}

\def\id{\text{id}}
\def\diag{\text{diag}}
\def\Gr{\text{Gr}}

\def\gl{\mathfrak{gl}}

\def\bba{\mathbb A}
\def\bbk{\mathbb K}
\def\bbz{\mathbb Z}

\def\bbf{\mathbb F}
\def\bbp{\mathbb P}
\def\bby{\mathbb Y}

\def\tsb{\textsf{B}}

\def\frakh{\mathfrak{H}}

\def\frakx{\mathfrak{X}}
\def\frakt{\mathfrak{T}}
\def\frakv{\mathfrak{V}}
\def\frako{\mathfrak{O}}

\def\cv{\mathcal {V}}

\def\cb{\mathcal {B}}

\def\cn{\mathcal {N}}
\def\cg{\mathcal {G}}

\def\b0{{\bar 0}}
\def\b1{{\bar 1}}

\def\ttt{{\mathfrak t}}

\def\ggg{{\mathfrak g}}
\def\hhh{{\mathfrak h}}

\def\bbb{{\mathfrak b}}

{\vskip-\lastskip\medskip
  \noindent
  {\em #1.}\enspace
  }%
{\qed\par\medskip
  }

\begin{document}
\title[Generic Property and Borel subalgbras]{Generic property and conjugacy classes of homogeneous Borel subalgebras of restricted  Lie algebras}
%{Generic property and conjugacies of Borel subalgebras for restricted  Lie algebras}

\author{Bin Shu}
\address{Department of Mathematics, East China Normal University, Shanghai 200241, China.}\email{bshu@math.ecnu.edu.cn}

%\author{Mengmeng Zhang}
%\address{Department of Mathematics, East China Normal University,
%Shanghai 200241,  China.} %\email{bshu@math.ecnu.edu.cn}

\thanks{MSC (2000 Revision) 17B 50, 17B 05, 17B 20. This work is  supported partially by the NSF of China (No. 11271130; 11671138),  Shanghai Key Laboratory of PMMP}

\begin{abstract} Let $(\ggg,[p])$ be a finite-dimensional restricted Lie algebra over an algebraically closed field $\bbk$
of characteristic $p>0$, and $G$ be the adjoint group of $\ggg$. We say that $\ggg$ satisfying the {\sl generic property} if $\ggg$ admits generic tori introduced in \cite{BFS}. A Borel subalgebra (or Borel for short) of $\ggg$ is by definition  a maximal solvable
 subalgebra containing a maximal torus of $\ggg$, which is further called generic if additionally containing a generic torus.
 %Then the maximal tori of any given Borel subalgebra are maximal in  $\ggg$ ( a maximal solvable algebra of $\ggg$ is called  a Borel subalgebra in the paper).
%Generic Borel subalgebras are by definition a class of Borel
%subalgebras containing a generic torus.
In this paper, we first settle  a conjecture proposed by Premet in \cite{Pr2}  on regular Cartan subalgebras of restricted Lie algebras. We prove that the statement in the conjecture for a given $\ggg$  is valid if and only if it is the case when $\ggg$ satisfies the generic property.
% And for
%generic Borel subalgebras, we show that there are some connection between them
%and the nilpotent cone of $\ggg$ (we expect some more interesting
%connections between generic Borel subalgebras and representations of
%$\ggg$).
%We further prove that all maximal solvable
%subalgebras of $\ggg$ are Borels whenever $p>\dim\ggg$.
We then classify
the conjugay classes of homogeneous Borel subalgebras of the restricted simple
Lie algebras $\ggg=W(n)$ under $G$-conjugation when $p>3$, and present
the representatives of these classes. Here $W(n)$ is the so-called Jacobson-Witt algebra, by definition the derivation algebra of the truncated polynomial ring $\bbk[T_1,\cdots,T_n]\slash (T_1^p,\cdots,T_n^p)$. We also describe the closed
connected solvable subgroups of $G$ associated with those
representative Borel subalgebras.
\end{abstract}

%\footnotetext[1]{\textit{\bf Keywords:}~~Algebraic supergroup;
%affine quotient; .}

%\footnotetext[2]{\textit{\bf 2000 AMS Subject Classifications:}
%17B50, 17B05.}

\maketitle

\section*{Introduction}
 In the structure theory of classical Lie algebras (which arise from  Lie algebras of connected semi-simple algebraic groups), Borel subalgebras (which arise from  Lie algebras of Borel subgroups)
 along with Cartan subalgebras, Weyl groups, {\sl etc} are very important. These usual structural aspects
 constitute most important parts of the machinery of classical Lie algebras, playing fundamental roles in the structure and representation theory of classical Lie algebras (cf. \cite{Hum}, \cite{Hum-1}, \cite{Jan2}, \cite{Jan2} and \cite{Jan3} {\sl{etc}}).  There the basic nature  is that for a connected semi-simple algebraic group $G$ and $\ggg=\Lie(G)$,  all Borel ({{resp}}. Cartan) subalgebras of $\ggg$ are conjugate under $G$ (note that we can identify  Borel subalgebras with maximal solvable subalgebra containing maximal tori when the ground field is an algebraically closed field of characteristic $\neq 2,3$ (cf.  \cite[\S{14.3}]{Hum-1} and \cite[\S{III.4}]{Seligman})). Especially, the variety of all Borel subalgebras plays some key role in the representations of $\ggg$ (cf. \cite{BMR}).
By contrast, Cartan subalgebras and the maximal solvable subalgebras containing maximal tori, which will be still called {\sl Borel subalgebras} in the present paper,  of arbitrary restricted Lie algebras in prime characteristic, usually are not longer conjugate.
In \cite{Pr1} and \cite{Pr2}, Premet studied systemically Cartan subalgebras of restricted Lie algebras. Regular Cartan subalgebras which are the most important class of Cartan subalgebras, by definition the ones containing a maximal torus of maximal dimension, are proved to be conjugate  by means of a finite number of so-called elementary switchings (some invertible linear transformations provided by root vectors). He further proposed a conjecture on the conjugation situation of regular Cartan subalgebras under the adjoint group $G$ of $\ggg$ (see Conjecture \ref{premetconj}).
As to Borel subalgebras, there is less study on them so far.
Especially, for non-classical restricted simple Lie algebras, we have no answer to the questions how many and what  the conjugacy classes of Borel subalgebras are, and what kind of role  the Borel subalgebras could play in the representations of the Lie algebras (for those Lie algebras, Cartan subalgebras have been well-known (cf. \cite{St}). Our motivation is to understand  more on Borel subalgebras of non-classical restricted simple Lie algebras, and to exploit some possible connection among the theory of Borel subalgebras,  representations and others.

The purpose of the present paper is twofold. One is to settle Premet's conjecture mentioned above, the other one is to study Borel subalgebras of restricted Lie algebras.

According to Block-Wilson-Strade-Premet classification of restricted simple Lie algebras over an algebraically closed field of characteristic $p>3$ (cf. \cite{PS}), it is known that aside from the analogues of the complex simple Lie algebras (called classical Lie algebras) there are usually four additional classes of restricted simple Lie algebras{\footnote{The case when $p=5$ involves the extra ones Melikian algebras.}}, the so-called restricted Lie algebras of Cartan type, among of which the Jacobson-Witt algebras $W(n)$ will be a main topic of the present paper. In his article \cite{Pr3}, Premet studied analogy of Weyl groups and of the Chevalley restriction theorem in the complex simple Lie algebras,  and the variety of nilpotent elements for $W(n)$. Along this direction, Bois, Farnsteiner and the author of the present paper developed some general theory of Weyl groups for restricted Lie algebras, and studied the Weyl groups for the other three classes of Cartan type Lie algebras (cf. \cite{BFS} and \cite{Fa}). They proposed in \cite{BFS} the notion  "generic tori" which plays the same important role as the maximal tori in classical Lie algebras, associated with which Weyl groups and the Chevalley restriction theorem were obtained, with aid of the classification of isomorphism classes of maximal tori in \cite{De1} and \cite{De2}.
Inspired by the work above-mentioned, we begin the study of Borel subalgebras for restricted  Lie algebras. One aim  of the present paper is to present the conjugacy classes of homogenous Borel subalgebras of $W(n)$ (see Definition \ref{homo Borel} for "homogeneous").
%\iffalse
Meanwhile, we especially stress on generic Borel subalgebras and the associated solvable subgroups for a non-classical restricted Lie algebra which could be expected to play some  important roles as Borel subalgebras and Borel subgroups do in classical cases. The special feature of generic Borel subalgebras sheds some new light on the further study of geometric aspects and modular representations of non-classical Lie algebras (to see \cite{O}).
%\fi

Our paper is organized as follows. In the first section, we will give the notion {\sl{generic property}} for a restricted Lie algebra. Then we will prove that for generic Cartan subalgebras (see \S\ref{Generic} for the definition), all of them are conjugate under the adjoint group of $\ggg$ (Theorem \ref{genericCT}). Consequently, it can be shown that the statement for a given restricted Lie algebra in Premet's conjecture on regular Cartan subalgebras holds if and only if it is the case when the restricted Lie algebra satisfies the generic property (Theorem \ref{correctpremetconj}). The notion of generic Borel subalgebras are proposed here.
In Section 2, we recall some knowledge  on $W(n)$, then  propose $\ttt_r$-grading for $W(n)$, associated with the representative  $\ttt_r$   of the $r$th conjugacy class of maximal tori, then we introduce the notion of  homogeneous Borel subalgebras of $W(n)$. Section 3 is devoted to demonstrating all standard Borel subspaces (whose definition will be seen in \S\ref{Borels}) and then to proving that those subspace are Lie subalgebras and maximal solvable ones.  In Section 4, we complete the arguments on classification of isomorphism classes of homogeneous Borel
 subalgebras of $W(n)$ (Theorem \ref{maintheorem}) when the characteristic  $p$ of the ground filed $\bbk$ is bigger than $3$.
 %and also complete the isomorphism classification of maximal solvable subalgebras of $W(n)$ when $p>n$.
In Section 5, we prove that the solvable groups associated with Borel subalgebras are connected, and further give a precise description of the solvable subgroups associated with the generic Borel subalgebras. %In the concluding section, we propose some future topics in our research \cite{OSX} (including the same topics for restricted Lie algebras of other Cartan types).

In this paper, we have to use many notations in the arguments. For the readers' convenience, we make a notation list  at the end of the paper.

\section{Generic property and Borel subalgebras}
Throughout, % we shall be working over an algebraically closed filed $\bbk$ of odd characteristic char$(\bbk)=p>0$. Unless mentioned otherwise,
all vector spaces are assumed to be finite-dimensional and over an algebraically closed filed $\bbk$ of prime characteristic char$(\bbk)=p>0$ unless mentioned otherwise. Given a restricted Lie algebra $(\ggg,[p])$, we have an {\sl{adjoint group}} $G:=\Aut_p(\ggg)^\circ$, the identity component of its restricted automorphism group.

\subsection{Borel subalgebras}\label{1.1} A maximal solvable subalgebra $\cb$ of a restricted Lie algebra $(\ggg, [p])$ is called a {\sl {Borel subalgebra}} (or {\sl{Borel}} for short) if $\cb$ contains a maximal torus of $\ggg$ (recall that a torus $\ttt$ is by definition an abelian restricted subalgebra consisting of semi-simple elements, i.e.  $X\in (X^{[p]})_p$ for all $X\in\ttt$, where $(X^{[p]})_p$ stands for the restricted subalgebra generated by $X^{[p]}$ (see \cite[\S2.3]{SF})). A Cartan subalgebra is called a regular one if it contains a torus of maximal dimension. %A Borel subalgebra is called a regular one if it contains a regular Cartan subalgebra.

%\subsection{ }
The following observation is clear.
\begin{lemma} \label{restrictedborel} Let $(\ggg,[p])$ be a restricted Lie algebra over $\bbk$. Then every maximal solvable subalgebra of $\ggg$ is a restricted subalgebra.
\end{lemma}
\begin{proof} Let $\cb$ be a given maximal solvable subalgebra of $\ggg$.
Note that the restricted subalgebra $\langle\cb\rangle_p$ generated by $\cb$ is still  solvable. By the maximal solvableness of Borel subalgebras, $\cb$ coincides with $\langle\cb\rangle_p$. Hence $\cb$ itself is restricted.
\end{proof}

\subsection{Generic property and Premet's conjecture}\label{Generic}
 For a restricted Lie algebra $(\ggg,[p])$, denote by $\mu(\ggg)$ the maximal dimension of all tori $\ttt\subset \ggg$. Following Premet, we say that $\hhh$ is a regular Cartan subalgebra of $\ggg$ if this Cartan subalgebra contains a maxiaml torus of dimensin $\mu(\ggg)$ (cf. \cite{Pr2}). The dimension of any regular Cartan subalgebra coincides with $r(\ggg)$ the rank of $\ggg$ (cf. \cite[Theorem 1]{Pr1}). For a given regular Cartan subalgebra $\hhh$ with a maximal torus  $\ttt$, one has $\ggg=\hhh\oplus \sum_{\alpha\in \Delta(\ggg,\hhh)} \ggg_\alpha$,  a root-space decomposition of $\ggg$ with respect to $\hhh$, where $\Delta(\ggg,\hhh)$ is the corresponding root system.  Then, any two regular Cartan subalgebras in $\ggg$ can be obtained from each other by means of a finite number of elementary switchings defined via the root-space decompositions mentioned above (cf. \cite[Theorem 1]{Pr2}).

 A torus $\ttt_\gen$ of $\dim\ttt_\gen=\mu(\ggg)$  is called generic if  $G\cdot\ttt_\gen$ is a dense subset of $\overline {\textsf{S}}_\ggg$, where $\textsf{S}_\ggg$ stands for the Zariski closure of all semisimple elements in $\ggg$ (cf. \cite[\S1]{BFS}). Set $\hhh_\gen:=C_\ggg(\ttt_\gen)$,  the centralizer of $\ttt_\gen$ in $\ggg$. By \cite[Theorem 2.4.1]{SF}, $\hhh_\gen$ is a regular Cartan subalgebra of $\ggg$, which is called a generic Cartan subalgebra.
A Borel subalgebra $\cb$ of $\ggg$ is called generic if it contains a generic torus of $\ggg$.

Recall that $X\in\ggg$ is called regular if the Fitting-nilspace $\ggg^0(\ad X):=\{v\in \ggg\mid \ad X^{m(X, v)}v=0 \mbox{ for some postive integer }m(X,v)\}$ has the minimal dimension among all Fitting-nilspaces when $X$ runs over $\ggg$ (cf. \cite{Pr1} and \cite{Pr2}). %In this case, $\ggg^0(\ad X)=C_\ggg(X_s)$ is a regular Cartan subalgebra for the Jorden-Chevalley-Seligman decomposition  $X=X_s+X_n$ where $X_s$ is semisimple and $X_n$ is $p$-nilpotent, with $[X_s,X_n]=0$ (cf. \cite[Theorem 2.3.5]{SF}, \cite[Lemma 4, Theorems 1 and 2]{Pr1}, \cite[Lemma 3.1(1)]{Fa}) and the comments in \cite[Page 4190]{Fa}).
 Denote by $\mbox{Reg}(\ggg)$ the set of regular elements of $\ggg$. Then $\mbox{Reg}(\ggg)$ is an open dense subset of $\ggg$.
A regular element $X$ in $\ggg$ is called generic if it happens that $G\cdot \ggg^0(\ad X)$ is a dense subset of $\ggg$. It is worthwhile reminding that it is not necessary for a restricted Lie algebra to contain generic elements. For the criterion of the existence of generic elements, we have the following lemma by the definitions and \cite[Proposition 1.7]{BFS}.

\begin{lemma} ({\bf{Definition}}) \label{genericlemma} Let $(\ggg,[p])$ be a restricted Lie algebra. We say that  $\ggg$ satisfies the generic property if $\ggg$ admits generic tori. The generic property is satisfied if and only if one of the following conditions is done.
\begin{itemize}
\item[(1)] There are generic Borel subalgebras in $\ggg$.
\item[(2)] There are generic Cartan subalgebras in $\ggg$.
\item[(3)] There are generic elements in $\ggg$.
\end{itemize}
\end{lemma}

\begin{example} \label{examplegen}
 (1) For a classical Lie algebra $\ggg$,  all regular elements are generic because the saturation $G\cdot \hhh$ is a dense subset of $\ggg$ for any Cartan subalgbra $\hhh$  (cf. \cite[Theorem III 4.1]{Seligman}). According to \cite[Proposition 3.3]{BFS}, restricted simple Lie algebras of type $W, S$ and $H$  have generic elements. So these Lie algebras  satisfy the generic property. However, the restricted contact Lie algebras do not satisfy the generic property (cf. \cite[Theorem 6.6]{BFS}).

(2) One can easily give a description of all generic elements for the Jacobson-Witt algebra $W(n)$. %(see Remark \ref{remarkongeneric}(5)).
\end{example}

If $\hhh_\gen$ is a generic Cartan subalgebra of $\ggg$ containing the generic torus $\ttt_\gen$, then it follows from \cite[Lemma 3.2]{Fa} that there exists an open dense subset $U$ in $\ttt_\gen$ such that all elements of $U$ are generic. Furthermore, we will see  in the forthcoming Corollary \ref{correctpremetconj} that in such a case, there exists an open dense subset $V$ in $\ggg$ such that all elements of $V$ are generic.

\begin{theorem}\label{genericCT} Let $(\ggg, [p])$ be a restricted Lie algebra  satisfying the generic  property.
Then all generic Cartan subalgebras (resp. generic tori) are conjugate under $G$.
%Hence Conjecture \ref{premetconj} holds in this case.
\end{theorem}

\begin{proof} For any given two generic tori Cartan subalgebras $\hhh_1$ and $\hhh_2$,  by the definition $G\cdot\hhh_i$, $i=1,2$ are two dense subsets of $\ggg$. Note that both $G\cdot \hhh_i$ are constructible, thereby $G\cdot \hhh_i$, $i=1,2$, contain open dense subsets of $\ggg$ respectively (cf. (cf. \cite[II.\S3]{Ha} or \cite[4.4]{Hum}). %Thus we have dominant morphisms $G\times \ttt_i\rightarrow \overline S_\ggg$, $i=1,2$.
On the other hand, the set $\text{Reg}(\ggg)$ of regular elements of $\ggg$ is also an open dense subset of $\ggg$, which is $G$-stable. So we have a fact that
$(G\cdot \hhh_1)\cap (G\cdot \hhh_2)\cap \text{Reg}(\ggg)$ is nonempty. Take an element $H$ from the above nonempty. Then there is a unique regular Cartan subalgebra $\hhh$ containing $H$. Hence $\hhh$ coincides with $g_1\cdot\hhh_1$ and $\ggg_2\cdot\hhh_2$ for some $g_1,g_2\in G$. Thus $\hhh_1$ and $\hhh_2$ are conjugate. We complete the proof.
\end{proof}

In \cite[Conjecture 2]{Pr2}, Premet proposed a conjecture:
\begin{conj}(Premet conjecture) \label{premetconj}
For a finite-dimensional restricted Lie algebra $\ggg$, there exists a nonempty Zariski open subset V consisting of regular elements and such that for any $u,v\in V$ the Cartan subalgebras $\ggg^0(\ad u)$ and $\ggg^0(\ad v)$ are conjugate under $G$.
\end{conj}
From the above theorem, we can prove that the statement in Premet's conjecture for a given $\ggg$ is valid if and only if $\ggg$ satisfies the generic property.

\begin{theorem} \label{correctpremetconj} Let $(\ggg,[p])$ be a restricted Lie algebra, and $G$ be the adjoint group of $\ggg$. Then the following statements are equivalent.
\begin{itemize}
\item[(1)] There exists a nonempty Zariski open subset V consisting of regular elements and such that for any $u,v\in V$ the Cartan subalgebras $\ggg^0(\ad u)$ and $\ggg^0(\ad v)$ are conjugate under $G$.
\item[(2)] The Lie algebra $\ggg$ satisfies the generic property.
\end{itemize}
\end{theorem}
\begin{proof} $(2) \Rightarrow (1) $: Suppose that $\ggg$ satisfies  the generic property. Then $\ggg$ has a generic Cartan subalgebra $\hhh_\gen$ (cf. Lemma \ref{genericlemma}). By the definition, $G\cdot \hhh_\gen$ is dense in $\ggg$, thereby a constructible subset of $\ggg$. This means that $G\cdot\hhh_\gen$ contains a open dense subset $U$. Denote $V=U\cap\mbox{Reg}(\ggg)$. Then $V$ is non-empty and an open dense subset of $\ggg$. For any $v\in V$, $\ggg^0(\ad v)$ is $G$-conjugate to $\ggg^0(\ad v_0)$ for $v_0\in \hhh_\gen\cap \mbox{Reg}(\ggg)$ with $v=g\cdot v_0$ for $g\in G$. From the uniqueness of the regular Cartan subalgebra containing $v_0$ it follows that $\ggg^0(\ad v_0)$ coincides with $\hhh_\gen$. Thus $\ggg^0(\ad v)$ is also a generic Cartan subalgebra. Thanks to  Theorem \ref{genericCT},  the Cartan subalgebra $\ggg^0(\ad v)$ and $\ggg^0(\ad u)$ are conjugate under $G$ for $u,v\in V$. The statement (1) follows.

$(1)\Rightarrow (2)$: Suppose that $V$ is a nonempty Zariski open subset of $\ggg$ consisting of regular elements,  and that the regular Cartan subalgebras $\ggg^0(\ad v)$ for all $v\in V$ are conjugate under $G$. Note that  $v\in \ggg^0(\ad v)$. Hence for any given $v\in V$, $\overline {G\cdot \ggg^0(\ad v)}\supset \overline V=\ggg$. Thus, the Cartan subalgebra $\hhh:=\ggg^0(\ad v)$ is generic (in particular, $v$ is a generic element of $\ggg$). The statement (2) follows.
\end{proof}

As said in Example \ref{examplegen}, restricted simple Lie algebras of type $W, S$ and $H$ satisfy the generic property. Hence the statement in Premet's conjecture holds for them. However,  the restricted contact simple Lie algebras does not satisfy the generic property. So the contact restricted simple Lie algebras are examples showing the failure of Premet's conjecture.

In view of Theorem \ref{genericCT}, we propose the following question.

\begin{remark} Let $(\ggg,[p])$ be a restricted Lie algebra satisfying the generic property. When are all generic Borel subalgebras  conjugate under the adjoint group $G$?

 In view of known examples, we hope that for a restricted Lie algebra $\ggg$ satisfying  the generic property, then all generic Borel subalgebras of maximal dimension are conjugate under $G$.
\end{remark}

\section{Automorphisms and standard maximal tori of  the restricted simple Lie algebra $W(n)$}
From now on, we assume that $\bbk$ is an algebraically closed field of characteristic $p\geqslant 3$.
\subsection{}\label{firstnotations}
Set $\bbp=\{0,1\cdots,p-1\}$. For an element ${\bf{a}}=(a_1,\cdots,a_n)\in \bbp^n$, we denote $|{\bf{a}}|:=a_1+\cdots+a_n$.
Define a truncated polynomial algebra $A(n)$ to be the quotient of the polynomial ring $\bbk[T_1,\cdots,T_n]$ by the ideal generated by $T_1^p,\cdots,
T_n^p$. Set $x_i$ to be the image of $T_i$ in the quotient. Then $A(n)=\sum_{{\bf{a}}\in\bbp^n}\bbk{\bf{x}}^{\bf{a}}$, where ${\bf{x}}^{\bf{a}}=x_1^{a_1}\cdots x_n^{a_n}$ with ${\bf{a}}=(a_1,\cdots,a_n)\in \bbp^n$. We sometimes write $A(n)$ as $\bbk[x_1,\cdots,x_n]$ to emphasize  those indeterminants, and naturally regard $A(n)=\bigoplus_{i=0}^{n(p-1)} A_{[i]}$ as a $\bbz$-graded algebra, where $A_{[i]}$ is  spanned by monomials ${\bf{x}}^{\bf{a}}=x_1^{a_1}\cdots x_n^{a_n}$, $|{\bf{a}}|=i$. The Jacobson-Witt algebra $W(n)$ is defined to be the derivation algebra of $A(n)$.
%\iffalse
This is to say, $W(n)$ consists of all linear transformations $D$ of $A(n)$ satisfying $D(fg)=D(f)g+fD(g)$ for  $f,g\in A(n)$. It is easily seen that $W(n)$ is a free $A(n)$-module of rank $n$ with basis $\partial_i$, $i=1,\cdots,n$,
$$W(n)=\sum_{i=1}^n A(n)\partial_i.$$
 Here $\partial_i$ is the image of ${\partial\over \partial T_i}$ in the quotient of the $\bbk[T_1,\cdots,T_n]$-module $\Der(\bbk[T_1,\cdots,T_n])$ by the submodule arising from the ideal  generated by $T_i^p,i=1,\cdots,n$. Hence
 \begin{align*}
\partial_i {\bf{x}}^{\bf{a}}=a_i{\bf{x}}^{{\bf{a}}-\epsilon_i}
\end{align*}
where ${\bf{\epsilon}}_i=(\delta_{i,1},\cdots,\delta_{i,n})\in \bbp^n$, $\delta_{i,j}=1$ if $i=j$, and $\delta_{i,j}=0$ otherwise.
%\fi

 Set $\ggg=W(n)$. Then $\ggg$ is a $\bbz$-graded restricted simple Lie algebra. The $\bbz$-grading of $W(n)$ arises from the one of the truncated polynomial ring  $\bbk[T_1,\cdots,T_n]$. More precisely, $\ggg=\sum_{i}\ggg_{[i]}$ with $\ggg_{[i]}=\sum_{j=1}^n A_{[i+1]}\partial_j$ and the following properties  hold:%We have further $\ggg^{[p]}_{[i]}\subset \ggg_{[pi]}$.
 %the following property
 \begin{align}\label{standardgrading}
 \ggg=\bigoplus_{i=-1}^h\ggg_{[i]}, [\ggg_{[i]},\ggg_{[j]}]\subset \ggg_{[i+j]},\ggg^{[p]}_{[i]}\subset \ggg_{[pi]}
 \end{align}
where $h=n(p-1)-1$, and we set $\ggg_{[i]}:=0$ if $i$ is not between $-1$ and $h$.
Associated with such a grading, one has a filtration
\begin{align}\label{filtation}
\ggg=\ggg_{-1}\supset \ggg_{0}\supset \cdots \supset \ggg_{h}\supset 0
\end{align}
 for $\ggg_i=\sum_{j=i}^h\ggg_{[j]}$, $i=-1,0,1,\cdots,h$.
For more details on $W(n)$, the readers can be referred to \cite[Ch.4]{SF}.

\subsection{Automorphisms of $W(n)$} \label{Automorphisms} Recall that an automorphism $\varphi\in\Aut(A(n))$ induces an automorphism $\overline\varphi$ of $W(n)$ defined via $\overline\varphi:D\mapsto \varphi\circ D\circ \varphi^{-1}$. The inducing correspondence gives rise to a group isomorphism from $\Aut(A(n))$ to $\Aut(W(n))$ (cf. \cite{Wilson}). As to the former, an automorphism is determined by its action on the generators $x_i$, $i=1,\cdots,n$. We have a criterion that an algebra endomorphism $\varphi$ of $A(n)$ is an automorphism if and only if $\varphi$ stabilizes the unique maximal ideal of $A(n)$ and  the determinant
$\det((\partial_i\varphi(x_j))_{n\times n})$ is invertible in $A(n)$.
\iffalse
Furthermore, $\Aut(W(n))=\Aut(W(n))_0\ltimes \Aut(W(n))_1$, where $\Aut(W(n))_0$ consists of all grading-preserving automorphisms, and $\Aut(W(n))_1$ stands for the unipotent radical. The former is isomorphic to $\GL(n,\bbk)$, which consists  of invertible linear transformations on the $\bbk$-vector space spanned by $\{x_1,\cdots,x_n\}$.
\fi
\begin{theorem} (\cite{Wilson})\label{wilsonauto} Let $\ggg=W(n)$ over $\bbk$ of characteristic $p\geqslant3$ (unless $n=1$ with assumption $p>3$). The following statements hold.
\begin{itemize}
\item[(1)] The automorphism group $\text{Aut}(\ggg)$ coincides with the adjoint group $G=\text{Aut}_p(\ggg)^\circ$ (see the paragraph  before  \S\ref{1.1} for the notation).% Hence it is is a connected algebraic group.
\item[(2)] The group $G$ is a semi-direct product $G=G_0\ltimes U$, where $G_0\cong \GL(n,\bbk)$ consists of those autormophisms preserving the $\bbz$-grading of $\ggg$, and $$U=\{g\in G;(g-\id_\ggg)(\ggg_i)\subset \ggg_{i+1}\}.$$
\end{itemize}
\end{theorem}

\subsection{Conjugacy classes of the maximal tori} \label{2.3}% and the standard maximal tori
According to Demu\u{s}kin's result \cite{De1}, we have the following conjugacy results for maximal tori of $W(n)$.
\begin{theorem}\label{Demuskintori} Let $\ggg=W(n)$. Then the following statements hold.
\begin{itemize}
\item[(1)] Two maximal tori $\ttt,\ttt'$ belong to the same $G$-orbit if and only  if $\dim\ttt\cap\ggg_0=\dim\ttt'\cap\ggg_0$.
\item[(2)] There are $(n+1)$ conjugacy classes of maximal tori of $\ggg$. Each maximal torus of $\ggg$ is conjugate to  one of   $$\ttt_r=\sum_{i=1}^n\bbk z_i\partial_i,\; r=0,1,\cdots,n$$
    where $z_i=x_i$ for $i=1,\cdots, n-r$, and $z_i=1+x_i$ for $i=n-r+1,\cdots,n$.
\end{itemize}
\end{theorem}
We call these $\ttt_r$ the standard maximal tori of $W(n)$.

\subsection{Gradings associated with $\ttt_r$}
Note that the  truncated polynomial algebra $A(n)$ can be presented as  the quotient algebra $\bbk[T_1,\cdots,T_n]\slash (T_1^p-1,\cdots,T_n^p-1)$. Denote the image $T_i$ by $y_i$ in the quotient algebra. Then we can write $A(n)$ as $\bbk[y_1,\cdots,y_n]$. Comparing with the notations used  in \S\ref{firstnotations},  we actually  have $y_i=1+x_i$, $i=1,\cdots,n$. More generally,  $A(n)$ can be presented as a truncated polynomial
$$\bbk[z_1,\cdots,z_{n-r};z_{n-r+1},\cdots,z_n]$$
 with generator $z_i:=x_i, z_j:=y_j, \; i=1,\cdots,n-r;j= n-r+1,\cdots,n$, and defining relations:
\begin{align*}
&[x_i,x_{i'}]=[y_{j},y_{j'}]=[x_i,y_j]=x_i^p=y_j^p-1=0.
\end{align*}
Thus, $W(n)$ can be  presented, as a vector space,
\begin{align}\label{W(n) r-decomp}
W(n)=\sum_{i=1}^n\sum_{{\bf{c}}(i)\in\bbp^n}\bbk{\bf{z}}^{{\bf{c}}(i)}\partial_i,
\end{align}
where ${\bf{z}}=(z_1,\cdots,z_n)$ and ${\bf{z}}^{{\bf{c}}(i)}=z_1^{c_1}\cdots z_n^{c_n}$ with ${\bf{c}}(i)=(c_1,\cdots,c_n)\in \bbp^n$. Associated to the presentation (\ref{W(n) r-decomp}), there is a $\bbz$-graded structure as below, called {\sl $\bbz(\ttt_r)$-grading}:
\begin{align} \label{W(n) trgraded}
W(n)=\bigoplus_s W^{(\ttt_r)}_{[s]}, \mbox{ with }
W^{(\ttt_r)}_{[s]}=\bbk\mbox{-Span}\{{\bf{z}}^{{\bf{c}}(i)}\partial_i \mid |{\bf{c}}(i)|=s+1, i=1,\cdots,n\}.
\end{align}
Actually, every homogenous space $W^{(\ttt_r)}_{[s]}$ is a $\ttt_r$-module.
%However, $W(n)$ is  not necessarily  a graded algebra with respect to the above gradation unless $r=0$.
For the case $r=0$, the associated graded structure in (\ref{W(n) trgraded}) is called a {\sl standard-graded structure}, coinciding with the one in (\ref{standardgrading}). It is worthwhile mentioning a fact that for a given $v\in W(n)$,  $v$ belongs to $W^{(\ttt_r)}_{[s]}$ %is an $s$-graded homogenous element of $W(n)$
if and only if
\begin{align*} %\label{Trhomo}
\ad T_r(v)=sv
\end{align*}
with $T_r:=\sum_{i=1}^n z_i\partial_i$, where $z_i=x_i$ and  $z_j=1+x_j$, $i=1,\cdots, n-r; j=n-r+1,\cdots,n$.

When talking about $\ttt_0$, we will omit the superscript for the associated graded structure as below
\begin{align*} %\label{W(n) tzerograded}
W(n)=\bigoplus_s W_{[s]}, \mbox{ with }
W_{[s]}=\bbk\mbox{-Span}\{{\bf{x}}^{{\bf{c}}(i)}\partial_i \mid |{\bf{c}}(i)|=s+1, i=1,\cdots,n\}.
\end{align*}
This gives rise to a $\bbz$-graded Lie algebra structure for $W(n)$ as shown in (\ref{standardgrading}). Thanks to Theorem \ref{wilsonauto}, the associated filtration is invariant under $\Aut(W(n))$. Let $\frakh$ be a subalgebra of $W(n)$.  Call $\frakh$ {\sl {a $\bbz(\ttt_r)$-graded subalgebra}} if $\frakh=\sum_{i}\frakh_{[i]}^{(\ttt_r)}$, where  $\frakh_{[i]}^{(\ttt_r)}=\frakh\cap W(n)^{(\ttt_r)}_{[i]}$.

We refine $\bbz(\ttt_r)$-grading.  %Actually, taking a toral basis $\{H_i:=z_i\partial_i\mid i=1,\cdots,n\}$ of $\ttt_r$,
%we  regard $\alpha=(\alpha_1,\cdots,\alpha_n)\in\ttt_r^*$ as an element of $\bbf_p^n$ (identifying $\bbf_p^n$ with a  subset of $\ttt_r^*$).
%The above proof shows that
If a subalgebra $\frakh$ is $\bbz(\ttt_r)$-graded, we set for $\alpha\in \bbp^n$ with $|\alpha|=i+1$,  $\frakh_\alpha=\{v\in\frakh^{(\ttt_r)}_{[i]}
\mid \ad z_i\partial_i(v)=\alpha_i v\}$. Then
\begin{align}\label{homo weights}
\frakh=\sum_{\alpha\in {\bbp^n}}\frakh_\alpha^{(\ttt_r)}.
\end{align}
In this sense, we call $\frakh$ is {\sl $\ttt_r$-graded}. The weight system of $\frakh$ associated with $\ttt_r$ is the set of all $\alpha$ with nonzero weight space $\frakh_\alpha^{(\ttt_r)}$.
Naturally, the conditions of bing $\ttt_r$-graded and of being $\bbz(\ttt_r)$ are equivalent.

% Summing up, we have
\iffalse
Here we naturally regard $\alpha_i\in\bbf_p$ as an element of $\bbp$, in a unique way. In this sense, we say that  $\frakh$ is {\sl{$\ttt_r$-graded}}. Furthermore, $\frakh=\sum_{i}\frakh_{[i]}^{(\ttt_r)}$ with $\frakh_{[i]}^{(\ttt_r)}=\sum_{\alpha\in\bbp^n,\alpha_1+\cdots+\alpha_n=i}\frakh_\alpha$.

 %We can refine $\bbz(\ttt_r)$-grading.
  Taking a toral basis $\{H_i:=z_i\partial_i\mid i=1,\cdots,n\}$ of $\ttt_r$,   we  regard $\alpha=(\alpha_1,\cdots,\alpha_n)\in\ttt_r^*$ as an element of $\bbf_p^n$ (identifying $\bbf_p^n$ with a  subset of $\ttt_r^*$).
 For a subalgebra $\frakh$ containing $\ttt_r$,  call $\frakh$ is $\ttt_r$-graded if
 \begin{align}\label{homo weights}
  \frakh=\sum_{\alpha\in {\bbp^n}}\frakh_\alpha^{(\ttt_r)},
   \end{align}
   with $\frakh_\alpha^{(\ttt_r)}=\{v\in\frakh\mid \ad H_i(v)=\alpha_i v\}$. Here we naturally regard $\alpha_i\in\bbf_p$ as an element of $\bbp$, in a unique way. In this sense, we say that  $\frakh$ is {\sl{$\ttt_r$-graded}}. Furthermore, we can write $\frakh=\sum_{i}\frakh_{[i]}^{(\ttt_r)}$ with $\frakh_{[i]}^{(\ttt_r)}=\sum_{\alpha\in\bbp^n,\alpha_1+\cdots+\alpha_n=i}\frakh_\alpha$.
\fi

\begin{defn}\label{homo Borel}
A subalgebra $Q$ is called  {\it homogeneous} if its image under $\varphi$ is $\ttt_r$-graded as long as $Q$ contains a maximal torus conjugate to $\ttt_r$ under $\varphi\in \Aut(W(n))$.
\end{defn}

\subsection{}\label{2.6}
Recall $W(n)_{[0]}\cong \gl(n,\bbk)$ under the mapping $x_i\partial_j\mapsto E_{ij}$, where $E_{ij}$ is the elementary matrix with  all entries equal to $0$ except $(i,j)$-entry equal to $1$. We then have a triangular  decomposition $W(n)_{[0]}=\sum_{i<j}\bbk x_j\partial_i +\ttt_0+\sum_{i<j}\bbk x_i\partial_j$. Denote $\bbb=\ttt_0+\sum_{i<j}\bbk x_i\partial_j$, which is a standard  Borel subalgebra of $W(n)_{[0]}$. We identify $W(n)_{[0]}$ with $\gl(n,\bbk)$, and $\Aut(W(n)_{[0]})$ with $\GL(n,\bbk)$ in the sequent arguments whenever the context is clear.

\begin{lemma}\label{The Third Key Lemma} Assume the characteristic $p$ of the ground field $\bbk$ is bigger than $3$, and  $W(n)_0=\sum_{i\geqslant 0}W(n)_{[i]}$.  The following statements hold.
\begin{itemize}
\item[(1)] All maximal tori in $W(n)_0$ are conjugate to $\ttt_0$ under $\Aut(W(n))$.
\item[(2)] All Borel subalgebras of $W(n)_0$ are conjugate to $\tsb_0:=\bbb+W(n)_1$ under $\Aut(W(n))$.
\end{itemize}
\end{lemma}
\begin{proof} (1) Set $\ggg=W(n)$， and $G=\Aut(W(n))$. Denote the standard-graded and filtered structure by $\ggg=\sum_i\ggg_{[i]}$ and $\{\ggg_i\}$ respectively. According to Dem\u{s}kin's result (cf. \cite{De1}), we only need to show that maximal tori of $\ggg_0$ are also the ones of $\ggg$. Suppose that $\ttt$ is a maximal torus of $\ggg_0$ with basis $Z_i, i=1,\cdots, m$. We can write $Z_i=T_i+V_i$ with $T_i\in \ggg_{[0]}$ and $V_i\in \ggg_1$. Note that $V_i$ is nilpotent. Hence $T_i$ must be a semisimple element. Thus the set $\{T_i, i=1,\cdots,m\}$ spans a maximal torus of $\ggg_0$, thereby the maximal torus of $\ggg_{[0]}$. The latter is conjugate to $\ttt_0$ under $\GL(n,\bbk)$ (cf. \cite[Theorem III.4.1]{Seligman}. The first assertion is proved.

(2) We first observe that for any Borel subalgebra $\bbb'$ of $\ggg_{[0]}\cong {\frak{gl}}(n,\bbk)$, $\cb':=\bbb'+\ggg_1$ must be a Borel subalgebra of $\ggg_0$. The solvableness of $\cb'$ comes from the fact that ${\cb'}^{(n)}\subset \ggg_1$, and $\ggg_1$ is nilpotent, where  $L^{(m)}$ for a Lie algebra $L$ denotes the $m$-th derived ideal, i.e.  $L^{(m)}:=[L^{(m-1)}, L^{(m-1)}]$ by induction  with $L^{(0)}:=L$. In order to check  the maximal solvableness of $\cb'$, we observe that $\cb'\subset \ggg_0$ is standard-graded with $\cb'_{[0]}=\bbb'$ and $\cb'_{[i]}=\ggg_{[i]}$ for all $i>0$. So any other solvable algebra of $\ggg_0$ containing $\cb'$ is also standard-graded. Then the maximal solvableness of $\bbb'$ in $\ggg_{[0]}$ implies that of $\cb'$ in $\ggg_0$. Note that any Borel subalgebra  of $\gl(n,\bbk)$ are conjugate to the standard one $\bbb$ under $\GL(n,\bbk)\subset \Aut(W(n))$ (cf. \cite[\S14.3-4]{Hum-1}). It follows that any Borel subalgebra of $\ggg_0$, up to conjugation,  can be constructed in this way.

By the invariance of the filtration $\{\ggg_i\}$ under $G$, we know those Borel subalgebras are conjugate to $\tsb_0=\bbb+W(n)_1$ under $\GL(n,\bbk)$. We complete the proof.
\end{proof}

%\begin{remark}\label{gzerorsolvable} According to \cite[Theorem D]{HS}, the arguments in the above proof shows that all maximal solvable subalgebras of $W(n)_0$ are Borel subalgebras when $p>n$.
%\end{remark}

\subsection{An application to  $W(1)$}   We assume $\ggg=W(1)$ and the characteristic $p$ of the ground field is bigger than $3$. In this special case,  we adopt the notations $x$ and $\partial$, with the same meaning as  $x_1$ and $\partial_1$ respectively. %By Lemma \ref{theSecondKeyLemma-1},
 It is easily known that  any subalgebra of $W(1)$ containing $\ttt_i$, $i\in\{0,1\}$ is $\ttt_i$-graded, thereby  all Borel subalgegbras of $W(1)$ is homogeneous. We can list the conjugation result of Borel subalgebras in $W(1)$ as below.

\begin{prop} There are only two conjugacy classes of Borel subalgebras for the Witt algebra $W(1)$. The standard representatives are
$\tsb_1:=\bbk\partial+\bbk x\partial$ and $\tsb_0:=\bbk x\partial+\bbk x^2\partial +\cdots +\bbk x^{p-1}\partial$.
\end{prop}

\begin{proof} Recall $W(1)$ admits two conjugacy classes of maximal tori, with representatives $\ttt_0=\bbk x\partial$, and $\ttt_1=\bbk(1+x)\partial$.  For a given Borel subalgebra $\cb$,
we can assume that up to conjugation, either $\cb\supset \ttt_0$, or
$\cb\supset \ttt_1$.

 In the case when $\cb\supset \ttt_0$, the homogeneous property  implies that $\cb$ is standard-graded. So $\cb=\sum_{i}\cb_{[i]}$, where $\cb_{[0]}=W_{[0]}$. If $\dim\cb_{[-1]}=1$,
  then the maximal solvable subalgebra $\cb\supset \bbk\partial+\bbk x\partial$.
      However, the latter is already a maximal solvable subalgebra.
       This is because
  any subalgebra  properly containing $\bbk\partial+\bbk x\partial$  must contain $\bbk \partial +\bbk x\partial +\bbk x^2\partial$, while
 % \begin{align}\label{BasicMax}
   the latter is not solvable. % \bbk \partial +\bbk x\partial +\bbk x^2\partial  \mbox{ is not solvable}.
  %\end{align}
  So it must happen that $\cb= \bbk\partial+\bbk x\partial$ in this case.
  If $\dim\cb_{[-1]}=0$, then $\cb\subset W(1)_{0}$. The latter is solvable. The maximal solvableness of $\cb$ implies $\cb=W(1)_0$. In the following argument, we assume that $\cb\supset \ttt_1$.

By the homogeneous property again, $\cb$ is $\ttt_1$-graded. Consider the $\bbz(\ttt_1)$-graded structure  $W(1)=\bigoplus_{i} W(1)^{(\ttt_1)}_{[i]}$.
In the  case $\dim \cb^{(\ttt_1)}_{[-1]}=1$, $\cb\supset \bbk\partial+\bbk(1+x)\partial=\bbk\partial+\bbk x\partial=\tsb_1$. Hence $\cb=\tsb_1$. Next we  consider the remaining  situation $\cb\supset \ttt_1$ and $\dim \cb^{(\ttt_1)}_{[-1]}=0$. In such a case, $\cb\subset \bbk(1+x)\partial+\sum_{i>0}\bbk(1+x)^i\partial$. As $\cb$ is a maximal solvable subalgebra,
$\cb\supsetneqq \bbk(1+x)\partial$. Therefore, there exists $D:=(1+x)^a\partial\in \cb$ with $p>a>1$ (note the homogeneous property). Consequently on can conclude that  $\cb$ coincides with $\bbk(1+x)\partial+\bbk(1+x)^a\partial$ since the latter is readily known  to be a maximal solvable subalgebra.  Actually,  any subalgebra properly containing $\bbk(1+x)\partial+\bbk(1+x)^a\partial$ must contain $\bbk(1+x)\partial+\bbk(1+x)^a\partial+\bbk \partial$ (note that $(1+x)^p=1$), which contradicts the assumption $\dim \cb^{(\ttt_1)}_{[-1]}=0$.  Next, we will show that $\cb$ ($=\bbk(1+x)\partial+\bbk (1+x)^a\partial$) is conjugate to $\tsb_1$. Consider $\varphi\in\Aut(W(1))$  which is defined via $\varphi(x)=(1+x)^b-1$, where $b\in\{ 2,\cdots,p-1 \}$ with $(1-a)b\equiv 1\mod p$. Then the inverse
 $\varphi^{-1}:x\mapsto (1+x)^{p-a+1}-1$. We have an automorphism $\overline\varphi$ of $W(1)$ induced by $\varphi$ (see \S\ref{Automorphisms}), denoted by $\Phi$. By a straightforward computation, we have
\begin{align}\label{compWitt}
&\Phi((1+x)^a\partial)=(p-a+1)\partial, \cr
&\Phi((1+x)\partial)=(p-a+1)(1+x)\partial.
\end{align}
Under the conjugation via such a $\Phi$, we have $\cb\cong \bbk\partial+\bbk(1+x)\partial=\tsb_1$.
\end{proof}

\begin{remark} For $\ggg=W(1)$, we can show that any maximal
solvable subalgebra is a Borel. Actually, for a given maximal
subalgebra $\cb$, it can be endowed with a filtration structure
$\{\cb_i\}$  inheriting the one of $\ggg$ as in (\ref{filtation}). Consider the graded subalgebra $\Gr(\cb)$ of $\ggg$. Then
$\Gr(\cb)=\sum_{i=-1}^{p-2}\Gr(\cb)_{[i]}$ is also a maximal
solvable subalgebra of $\ggg$. Now $\Gr(\cb)$ is normalized by $\bbk
x\partial$. Hence, $\Gr(\cb)$ contains $\bbk x\partial$. This
implies that $\cb$ contains  nonzero semi-simple elements, and then
contains some maximal torus of $\ggg$. Therefore,  the above
proposition covers the main result of \cite{YC}.

The above arguments of classifying Borels for $W(1)$  will be extended to the general case in the next sections. The ideas are essentially the same.
\end{remark}

%\section{Borels in $W(n)$}

\section{Standard  Borel subalgebras}

We maintain the notations as before.

\subsection{Borel subspaces} \label{Borels} We will first introduce  $(n+1)$ vector subspaces $\tsb_q$ in $W(n)$, $q=0,1,\cdots,n$.
Recall $\tsb_0=\bbb+W(n)_1$, where $W_1=\sum_{i\geqslant 1}W(n)_{[i]}$.
Next we set
$$\tsb_n=W(n)_{[-1]}+\bbb+ \sum_{q=1}^n\sum_{{\bf{a}}(q)}\bbk{\bf{x}}^{{\bf{a}}(q)} \partial_q,$$
where ${\bf{x}}=(x_1,\cdots,x_n)$ and ${\bf{a}}(q):=(a_1,\cdots,a_{q},0,\cdots,0)\in \bbp^{n}$ with $|{\bf{a}}(q)|>1$, $a_i$ runs through $\bbp$ for $i=1,\cdots,q-1$, and $a_q=0, 1$. We call $\tsb_0$ and $\tsb_n$ {\sl the nought-switched Borel subspace, and the full-switched Borel subspace} respectively.

Before introducing  general Borel subspaces $\tsb_q$, we need some conventions as below (some have appeared before, but now  presented formally).
\begin{conventions}\label{conventions}
Let $z_i$ be either $x_i$ or $(1+x_i)$, $i=1,\cdots,n$.
For  a subsequence ${\bf{u}}=(u_1,\cdots,u_q)$ of the sequence $(z_1,\cdots,z_n)$， i.e, $\{u_1,\cdots,u_q\}\subset\{z_1,\cdots,z_n\}$
we adopt the notations
\begin{itemize}
\item[(1)] Set $\tsb_0(u_1,\cdots,u_q)$ and $\tsb_q(u_1,\cdots,u_q)$ to be the nought-switched Borel subspace, and the full-switched Borel subspace of $W(q)$ respectively. Here $W(q)$ is the derivation algebra of $A(q)=\bbk[u_1,\cdots,u_q]$, which is the subalgebra of the truncated polynomial algebra $A(n)=\bbk[x_1,\cdots,x_n]$ generated by $u_1,\cdots,u_q$. For distinguishing indeterminants,  we sometimes need the notations $A(u_1,\cdots,u_q)$ and $W(u_1,\cdots,u_q)$ standing for those $A(q)$ and $W(q)$ respectively.
\item[(2)] Set $\ttt_0(u_1,\cdots,u_q):=x_{i_1}\partial_{i_1}+\cdots+x_{i_q}\partial_{i_q}$, and
$\ttt_q(u_1,\cdots,u_q)=(1+x_{i_1})\partial_{i_1}+\cdots+(1+x_{i_q})\partial_{i_q}$
when $(u_1,\cdots,u_q)=(x_{i_1},\cdots,x_{i_q})$. Those $\ttt_0$ and $\ttt_q$ mean the first and last standard tori of $W(x_{i_1},\cdots,x_{i_q})$ respectively.
\item[(3)] Set ${\bf{u}}^{\bf{a}}:=u_1^{a_1}\cdots u_q^{a_q}$
if ${\bf{a}}=(a_1,\cdots,a_q)\in \bbp^q$.
\end{itemize}
\end{conventions}
 It is the position for us to define the $q$th Borel subspace $\tsb_q$.  Let us first take ${\bf{u}}=(x_1,\cdots,x_{n-q})$ and ${\bf{w}}=(x_{n-q+1},\cdots, x_n)$. Define
$$\tsb_q=\tsb_{0}(x_1,\cdots,x_{n-q})+Q_q +\tsb_{q}(x_{n-q+1},\cdots,x_n),$$
where
\begin{align}\label{Qq}Q_q=\sum_{i=1}^{n-q}\sum_{{\bf{a}}(i),{\bf{b}}(i)}\bbk{\bf{u}}^{{\bf{a}}(i)}{\bf{w}}^{{\bf{b}}(i)} \partial_i +\sum_{j=n-q+1}^n\sum_{{\bf{a}}(j);{\bf{b}}(j)}\bbk{\bf{u}}^{{\bf{a}}(j)}{\bf{w}}^{{\bf{b}}(j)} \partial_j,
\end{align}
where ${\bf{a}}(i):=(a_1,\cdots,a_{n-q})\in \bbp^{n-q}$ is subjected to the condition that either $|{\bf{a}}(i)|>1$ or $|{\bf{a}}(i)|=1=a_1+\cdots+a_i$, while ${\bf{a}}(j)\in \bbp^{n-q}$ is subjected to the condition $|{\bf{a}}(j)|>0$, and ${\bf{b}}(-):=(b_{n-q+1},\cdots,b_n)$ runs through $\bbp^{q}$ for $(-)=(i), (j)$, subjected to the condition $|{\bf{b}}(i)|>0$.

In the sequent subsections, we will prove that all $\tsb_q$ are Borel subalgebras of $W(n)$.

\subsection{Switched positive root systems of the rigid root system} Set $E=\{-1, 0\}$. Consider a subset $\Delta$ of $\Gamma:=\bbp^n\times E^n$:
$$\Delta=\{(a_1,\cdots,a_n)\times (\eta_1,\cdots,\eta_n)\in\Gamma\mid \eta_1+\cdots+\eta_n=-1\},$$
and then set
$$\bar\Delta=\Delta\cup \{\infty\}.$$
For $\alpha=(\sum_ia_i\epsilon_i)\times(-\epsilon_s)$ and $\beta=(\sum_ib_i\epsilon_i)\times(-\epsilon_t)\in \Delta$, we define an ordered operator "$+: \bar\Delta\times \bar\Delta\rightarrow \bar\Delta$" as $\alpha+\beta:=(\sum_i(a_i+b_i)\epsilon_i-\epsilon_s)\times(-\epsilon_t)$ if it lies in $\Delta$, and $\alpha+\beta:=\infty$ otherwise; and define $\infty+\mbox{any}=\mbox{any}+\infty=\infty$.
Then $W(n)$ can be decomposed into a direct sum of one-dimensional root spaces, associated with the so-called rigid root system as below:
$$W(n)=\sum_{\alpha\in\bar\Delta}W(n)_\alpha$$
where $W(n)_\alpha:=\bbk x_1^{a_1}\cdots x_n^{a_n}\partial_j$ for $\alpha=(a_1,\cdots,a_n)\times (-\epsilon_j)$, and $W(n)_\infty:=0$.
We call $\bar\Delta$ the {\sl rigid root system} of $W(n)$. Then we have
\begin{align}\label{productformula}
[W(n)_\alpha,W(n)_\beta]\subset W(n)_{\alpha+\beta}+W(n)_{\beta+\alpha}.
\end{align}
Associated with the naught-switched, and full-switched Borel subspaces, we define positive root systems $\bar\Delta(0)_+$, and $\bar\Delta(n)_+$ respectively as below
\begin{align*}
\bar\Delta(0)_+:=&\{\epsilon_i\times(-\epsilon_j)\mid 1\leqslant i\leqslant j\leqslant n\}\cup\{\infty\}\cup \cr
&\{\sum_{i=1}^n a_i\epsilon_i\times (-\epsilon_j)\mid a_1+\cdots+a_n>1, j=1,\cdots,n\};\cr
\bar\Delta(n)_+:=&\{0\times(-\epsilon_j)\mid j=1,\cdots,n\}
\cup\{\epsilon_i\times(-\epsilon_j)\mid 1\leqslant i\leqslant j\leqslant n\}\cup\{\infty\}\cup\cr
& \{\sum_{i=1}^n a_i\epsilon_i\times (-\epsilon_j)\mid  a_1+\cdots+a_n=a_1+\cdots +a_j>1; a_j=0,\mbox{ or }1;j=1,2,\cdots,n\}.
\end{align*}
Then $\tsb_0=\sum_{\alpha\in \bar\Delta(0)_+}W(n)_\alpha$ and $\tsb_n=\sum_{\alpha\in \Bar\Delta(n)_+}W(n)_\alpha$. Generally, we set
$$\Bar\Delta(q)_+=\Bar\Delta^{(x_1,\cdots,x_{n-q})}(0)_+\cup\Bar\Delta^{(x_{n-q+1},\cdots,x_n)}(q)_+\cup \Bar\Delta\{Q\}_+,$$
where $\Bar\Delta^{(x_1,\cdots,x_{n-q})}(0)_+$, and $\Bar\Delta^{(x_{n-q+1},\cdots,x_n)}(q)_+$ denote the nought-switched positive root system and full-switched positive root system associated with $\bbk[x_1,\cdots,x_{n-q}]$ and $\bbk[x_{n-q+1},\cdots,x_n]$ respectively; and $\Bar\Delta\{Q\}_+$ denotes a subset of $\Bar\Delta_+$ as below:
\begin{align*}
\bar\Delta\{Q\}_+=&\{(\sum_{i=1}^{n-q}a_i\epsilon_i+\sum_{j=n-q+1}^nb_j\epsilon_j)\times(-\epsilon_r)\mid 1\leqslant r\leqslant n-q,
\mbox{either } a_1+\cdots+a_{n-q}>1 \cr
&\mbox{ or }   a_1+\cdots+a_{n-q}=a_1+\cdots+a_r=1; b_{n-q+1}+\cdots +b_n>0  \}\cup\{\infty\}\cup\cr
&\{(\sum_{i=1}^{n-q}a_i\epsilon_i+\sum_{j=n-q+1}^nb_j\epsilon_j)\times(-\epsilon_s)\mid n-q+1\leqslant s\leqslant n,
a_1+\cdots+a_{n-q}>0\}.
\end{align*}
By the above construction, we have
\begin{lemma} \label{subalgebras} Maintain the notations as above.
The following statements hold.
\begin{itemize}
\item[(1)] $\tsb_q=\sum_{\alpha\in \Bar\Delta(q)_+}W(n)_\alpha$, $q=0,1,\cdots,n$.
\item[(2)] $\tsb_q$ contains the maximal tori $\ttt_r$, $r=0,1,\cdots,q$.
\item[(3)] Furthermore, $\tsb_q$ is a restricted $\ttt_r$-module for $r=0,1,\cdots,q$. As a $\ttt_0$-module the weight system is described as below, in the sense
of (\ref{homo weights}):
$$\{{\bf{a}}=(a_1,\cdots,a_n) \in \bbf_p^n\mid  a_i\mbox{ coinicides with the } i\mbox{th entriy of }\overline\alpha \mbox{ for }\alpha\in \Delta(q)_+\}$$
where $\overline\alpha=(\cdot)+(\cdot\cdot)$ for $\alpha=(\cdot)\times(\cdot\cdot)$ and $\Delta(q)_+=\bar\Delta(q)_+\backslash \{\infty\}$.
\item[(4)] $\tsb_q$ is a subalgebra of $W(n)$.
\end{itemize}
\end{lemma}

\begin{proof} The first three statements  are clear. In order to prove the last one, we only need to check the operator "$+$" is closed in $\Bar\Delta(q)_+$  in view of  the first three statements, and Formula (\ref{productformula}), along with the homogenous condition. By a direct computation, we easily know that it is closed in $\Bar\Delta(0)_+$ and $\Bar\Delta(n)_+$. As to an arbitrary $q$, we note that in
$$\tsb_q=\tsb_{0}(x_1,\cdots,x_{n-q})+Q_q +\tsb_{q}(x_{n-q+1},\cdots,x_n),$$
the first and last summands are already known to be subalgebras, mutually orthogonal (i.e. the Lie brackets between them are  $0$). Furthermore, $Q_q$ is normalized by the first summand, and $[Q_q,\tsb_{q}(x_{n-q+1},\cdots,x_n)]\subset \tsb_{0}(x_1,\cdots,x_{n-q})+Q_q$.
 So we only need to prove that $Q_q$ is a subalgebra. For this, it is sufficient to check that the operator "$+$" is closed in $\Bar\Delta\{Q\}_+$. By a straightforward and easy computation, "$+$" is closed there.
\end{proof}

\subsection{Standard Borel subalgebras}\label{standardborels} In the concluding subsection, we prove all $\tsb_q$ are Borel subalgebras of $W(n)$.
\begin{prop} Maintain the notations as before.
 All Borel subspaces $\tsb_r$, $r=0, 1,2,\cdots,n$, are Borel subalgebras of $W(n)$, called the standard Borel subalgebras.
\end{prop}
\begin{proof} By the definition of Borel subspaces, every Borel subspace contains the maximal torus $\ttt_0$ of $W(n)$.
 According to Lemma \ref{subalgebras}(4), those Borel subspaces are subalgebras of $W(n)$. What we need to do  are the following two,
 one is to verify the solvableness of Borel subspaces, and the other one is to do for the maximality.

(1) We want to prove that all the subalgebras $\tsb_q$ are solvable. %We only need to prove that $\tsb_q^{(m)}=0$ for some positive integer $m$.
By Lemma \ref{The Third Key Lemma}, we know $\tsb_0$ is a solvable subalgebra.  As to $\tsb_n$, we first observe that $\tsb_n^{(n+1)}\subset \tsb_0$, from which it follows that $\tsb_n^{(n+1)}$ is solvable. Hence $\tsb_n$ itself is solvable.

For an arbitrary  $q$, recall $\tsb_q=\tsb_{0}(x_1,\cdots,x_{n-q})+Q_q +\tsb_{q}(x_{n-q+1},\cdots,x_n)$. The first and last summands have been known solvable subalgebras. Note that the middle summand subalgebra $Q_q$ is normalized by the first one and that the sum of the first and second summands are normalized by the third one.
\iffalse
$$[Q_q,\tsb_{q}(x_{n-q+1},\cdots,x_n)]\subset \tsb_{0}(x_1,\cdots,x_{n-q})+Q_q.$$
\fi
So we only need to verify that $Q_q$  is solvable. This follows since   $Q_q^{(q)}\subset W(n)_1$, the latter of which is solvable.

(2) We want to prove that  all  solvable subalgebras $\tsb_q$ are maximal. %not contained properly in any solvable subalgebras.

For $\tsb_0=\bbb+W(n)_1$, the maximality is  easily seen. Actually, if a solvable subalgebra $\frakh\supseteq \tsb_0$, then $\frakh$ is $\ttt_0$-graded by the homogenous condition. Then $\frakh=\sum_{q}\frakh_{[q]}$ with $\frakh_{[q]}=\frakh\cap W(n)_{[q]}$. Note that the subalgebra generated by $\partial_i,x_i\partial_i,x_i^2\partial_i$ is not solvable. So $\frakh_{[-1]}$ must be zero.  Hence $\frakh\subset W(n)_0$. On the other hand, $\frakh_{[0]}$ is solvable in $W(n)_{[0]}$. However $\bbb\subset \frakh_{[0]}\subset W(n)_{[0]}$, and $\bbb$ is already a Borel subalgebra of $W(n)_{[0]}$. Hence $\frakh_{[0]}$ coincides with $\bbb$. Thus $\frakh=\tsb_0$.% The maximality of $\tsb_0$ is true.

As to $\tsb_n$, recall
$$\tsb_n=W(n)_{[-1]}+\bbb+ \sum_{q=1}^n\sum_{{\bf{a}}(q)}\bbk{\bf{x}}^{{\bf{a}}(q)} \partial_q$$
with the notation ${\bf{x}}=(x_1,\cdots,x_n)$, ${\bf{a}}(q)=(a_1,\cdots,a_{q},0,\cdots,0)\in \bbp^{n}$ with $|{\bf{a}}(q)|>1$, $a_i$ runs through $\bbp$ for $i=1,\cdots,q-1$, and $a_q=0, 1$.
Suppose there is a solvable algebra $\frakh\supseteq \tsb_n$. Similar to the above argument, we  have that  $\frakh=\sum_{q}\frakh_{[q]}$ with $\frakh_{[q]}=\frakh\cap W(n)_{[q]}\supset (\tsb_n)_{[q]}$ with $\frakh_{[-1]}=W(n)_{[-1]}$ and $\frakh_{[0]}=\bbb$. If $\frakh$ contains $\tsb_n$ properly, then there must be a non-zero element $X\in\frakh_\alpha\subset\frakh_{[s]}$ for $s>0$ which is not in $(\tsb_n)_{[s]}$. We can write
$$X=\sum_{q=1}^n \sum_{{\bf{b}}(q)}C_{{\bf{b}}(q)}{\bf{x}}^{{\bf{b}}(q)} \partial_q$$
where  $C_{{\bf{b}}(q)}\in \bbk$, ${\bf{b}}(q)=(b_1,\cdots,b_n)\in \bbp^n$ satisfying either $(b_{q+1},\cdots,b_n)\neq 0$, or $(b_{q+1},\cdots,b_n)=0$ with $b_q>1$,  as long as $C_{{\bf{b}}(q)}\neq 0$. Say,  either  $b_{t}\neq 0$ for some $t>q$, or  $b_q=d\geq 2$.
By applying suitable adjoint actions $\ad D$ for some  $D\in \tsb_{[-1]}$,  one of the following situations happens:
either $\frakh$ contains a subalgebra generated by $x_t\partial_q, x_q\partial_t, x_t\partial_t-x_q\partial_q$, or $\frakh$ contains a subalgebra generated by $x_q^2\partial_q,x_q\partial_q, \partial_q$. But either subalgebra of  two kinds is not solvable, which contradicts  the solvableness of $\frakh$. Hence, $\frakh$ must coincide with $\tsb_n$. We complete the proof of the maximality of $\tsb_n$.

Let us finally investigate %the maximality of general Borel subspaces
$\tsb_q$.
%$=\tsb_{0}(x_1,\cdots,x_{n-q})+Q_q +\tsb_{q}(x_{n-q+1},\cdots,x_n)$.
 We need to prove that for any solvable subalgebra $\frakh$, the inclusion $\frakh\supseteq \tsb_q$ implies the equality $\frakh=\tsb_q$.
Observe
 %the containing relation of two solvable subalgebras in $W(x_1,\cdots,x_{n-q})$
 that $\frakh\cap W(x_1,\cdots,x_{n-q})\supseteq \tsb_0(x_1,\cdots,x_{n-q})$, along with a result just proved that   $\tsb_0(x_1,\cdots,x_{n-q})$ is a Borel subalgebra of $W(x_1,\cdots,x_{n-q})$. It follows that $\frakh\cap W(x_1,\cdots,x_{n-q})= \tsb_0(x_1,\cdots,x_{n-q})$. By the same reason, we have $$\frakh\cap W(x_{n-q+1},\cdots,x_{n})= \tsb_q(x_{n-q+1},\cdots,x_{n}).$$
 Thus we have
$$  \tsb_q\subseteq \frakh \subseteq \tsb_{0}(x_1,\cdots,x_{n-q})+\overline{Q_q} +\tsb_{q}(x_{n-q+1},\cdots,x_n)$$
where $\overline{Q_q}=\sum_{i=1}^{n-q}\sum_{{\bf{a}}(i),{\bf{b}}(i)}\bbk{\bf{u}}^{{\bf{a}}(i)}{\bf{w}}^{{\bf{b}}(i)} \partial_i +\sum_{j=n-q+1}^n\sum_{{\bf{a}}(j);{\bf{b}}(j)}\bbk{\bf{u}}^{{\bf{a}}(j)}{\bf{w}}^{{\bf{b}}(j)} \partial_j$,
 ${\bf{u}}=(x_1,\cdots,x_{n-q})$ and ${\bf{w}}=(x_{n-q+1},\cdots,x_n)$ , and  ${\bf{a}}(-):=(a_1,\cdots,a_{n-q})$ run through $\bbp^{n-q}$ for $(-)=(i),(j)$, and ${\bf{b}}(-):=(b_{n-q+1},\cdots,b_n)$ run through $\bbp^{q}$ for $(-)=(i),(j)$, subjected to the condition that $|{\bf{a}}(j)|>0$ and $|{\bf{b}}(i)|>0$.

Suppose $\frakh\supsetneqq \tsb_q$. Then there must be some element $X\in \overline{Q_q}\backslash Q_q$. Comparing $\overline{Q_q}$ and $Q_q$, we have
\begin{align*}%\label{violatingX-1}
X=\sum_{i=1}^{n-q}\sum_{{\bf{a}}(i),{\bf{b}}(i)}C_{{\bf{a}}(i),{\bf{b}}(i)}{\bf{u}}^{{\bf{a}}(i)}{\bf{w}}^{{\bf{b}}(i)} \partial_i
\end{align*}
with  $C_{{\bf{a}}(i),{\bf{b}}(i)}\in\bbk$, and with ${{\bf{a}}(i),{\bf{b}}(i)}$ violating the condition listed below (\ref{Qq}) as long as $C_{{\bf{a}}(i),{\bf{b}}(i)}\neq 0$. This is to say, there exists $j$: $1\leqslant i<j \leqslant n-q$ satisfying  $|{\bf{a}}(i)|=1=a_j$ while   $C_{{\bf{a}}(i),{\bf{b}}(i)}\neq 0$. By the same arguments as previously, we conclude that $\frakh$ contains $x_j\partial_i, x_i\partial_j$ and $x_i\partial_i-x_j\partial_j$, thereby contains the subalgebra generated by them, which is not solvable. This contradicts  the hypothesis that $\frakh$ is solvable. Consequently, we must have   $\frakh=\tsb_q$. Thus, we complete the proof of the maximality of the solvable subalgebras $\tsb_q$.% Summing up the two claims, we complete the proof.
\end{proof}

\section{Conjugacy classes of homogeneous Borel subalgebras}
Maintain the notations and conventions as before. Especially, the characteristic $p$ of the ground field $\bbk$ is assumed to be bigger than $3$ throughout this section.
\subsection{}\label{r B} In this section, we will prove that the standard  Borel subalgebras are representatives of conjugacy classes of all homogeneous Borel subalgebras. Let $\cb$ be any given Borel subalgebra of $W(n)$. By the definition of Borel subalgebras, $\cb$ contains a maximal torus. In view of Theorem \ref{Demuskintori}, we introduce an invariant $\textsf{r}(\cb)$ as follows
$$\textsf{r}(\cb):=\max\{r\mid \mbox{ there exists }\sigma\in \Aut(W(n))\mbox{ suth that } \sigma(\ttt_r)\subset \cb\}$$
 associated with the conjugacy class of $\cb$.
The following conclusion is obvious.
\begin{lemma}
$\textsf{r}(\tsb_r)=r$.
\end{lemma}\label{rBequalto0}
\subsection{} Let us begin the arguments with two special cases $\textsf{r}(\cb)=0$ and $\textsf{r}(\cb)=n$.
\begin{lemma}
Assume $\cb$ is a homogeneous  Borel subalgebra of $W(n)$ with $\textsf{r}(\cb)=0$.  Then $\cb\cong \tsb_0$.
\end{lemma}
\begin{proof} Up to conjugation, we may  assume $\cb\supset \ttt_0$. By the homogeneous condition, $\cb$ admits a standard-graded structure. So $\cb=\sum_i\cb_{[i]}$. We claim that $\dim\cb_{[-1]}=0$. Actually, if there is a non-zero $D\in \cb_{[-1]}$ which is expressed as $D=\sum_{i}c_{t_i}\partial_{t_i}$ with $c_{t_i}\in \bbk\backslash\{0\}$ for $t_i\in\{1,\cdots,n\}$. Those $\partial_{t_i}$ must fall in $\cb$ because $[x_{t_i}\partial_{t_i},D]=-c_{t_i}\partial_{t_i}\in \cb$. Hence $(1+x_{t_i})\partial_{t_i}\in\cb$, which contradicts the assumption  $\textsf{r}(\cb)=0$. Therefore, $\cb$ falls in $W(n)_0$. The remaining thing can be done, thanks to Lemma \ref{The Third Key Lemma}.
\end{proof}

\subsection{}
Let $\cb$ be a homogeneous Borel subalgebra of $W(n)$ with $\textsf{r}(\cb)=n$. This is to say, $\cb\supset \sigma(\ttt_n)$ for some $\sigma\in \Aut(W(n))$ and $\ttt_n=\bbk(1+x_1)\partial_1+\cdots+\bbk(1+x_n)\partial_n$.
\begin{lemma}\label{fullborel} $\cb$ is conjugate to $\tsb_n$.
\end{lemma}

\begin{proof} We might as well assume $\cb\supset \ttt_n$ without loss of generality.
By the homogeneous condition,  $\cb$ is $\ttt_n$-graded.  We will prove the lemma by induction on $k':=n-\dim\cb^{(\ttt_n)}_{[-1]}$ (recall the notation in \S2.5). We additionally set $k=\dim\cb^{(\ttt_n)}_{[-1]}$.

When $k'=0$, i.e. $\dim \cb^{(\ttt_n)}_{[-1]}=n$, we have  $\cb\supset \ttt_0$. By the homogeneous condition, $\cb$ is standard graded, i.e. $\cb=\sum_{i}\cb_{[i]}$. Furthermore, %further know that %Among the homogeneous subspaces of $\cb$,
$\cb_{[-1]}=W(n)_{[-1]}$ and $\cb_{[0]}\supset \ttt_0$. Set $\cb_0=\sum_{i\geq 0} \cb_{[i]}$. We claim that
 \begin{itemize}
\item[($\ast$)] Under the conjugation arising from some  permutation $\pi$ of $(12\cdots n)$ (see Theorem \ref{wilsonauto}(1)),
$$\cb_0\subset \sum_{i}\sum_{{\bf{a}}(i)}\bbk {\bf{x}}^{{\bf{a}}(i)}\partial_{\pi(i)}$$
where ${\bf{x}}=(x_{\pi(1)},\cdots,x_{\pi(n)})$, and ${\bf{a}}(i)=(a_1,\cdots,a_i,0,\cdots,0)\in \bbp^{n}$ subjected to the condition that $a_i=0$ or $a_i=1$.
\end{itemize}
Actually, if the claim ($\ast$) is not true, by the $\ttt_0$-graded property and by  suitable  $\ad \partial_s$-actions, $s\in \{1,\cdots,n\}$, we readily conclude that one of the following things happens:
  \begin{itemize}
  \item[($\dagger$)]
  there exist in $\cb$ a pair of elements $x_i\partial_j$ and $x_j\partial_i$ with $i\ne j$.
    \item[($\ddagger$)] there exists  $x_i^2\partial_i$ in $\cb$ for some $i$.
  \end{itemize}
The former means that there are a triple in $\cb$: $x_i\partial_j$, $x_j\partial_i$ and $x_i\partial_i-x_j\partial_j$; the latter means that there are a triple in $\cb$: $\partial_i$, $x_i\partial_i$ and $x_i^2\partial_i$.  Any  of both situations  contradicts the solvableness of $\cb$. Hence the claim ($\ast$) is true. Thus, by a suitable permutation $\tau$ of $(x_1,\cdots,x_n)$ which gives rise to an automorphism of $W(n)$,  $\cb$ is isomorphic to a solvable subalgebra of $\tsb_n$. Hence $\cb\cong \tsb_n$ by the maximal solvableness of Borel subalgebras. We complete the proof for $r=0$.

Now suppose $k'>0$, and the statement has held when less than $k'$.
Without loss of generality, we assume $\cb_{[-1]}^{(\ttt_n)}=\bbk\partial_1+\cdots+\bbk\partial_{k}$ (we may apply some $g\in \GL(n)$ to get this). Then $\cb\supset \ttt_{n-k}=\bbk x_1\partial_1
+\cdots+\bbk x_k\partial_k+\bbk(1+x_{k+1})\partial_{k+1}+\cdots+\bbk (1+x_n)\partial_n$. Set $\cb_1:=\cb\cap W(x_1,\cdots,x_k)$ and
$\cb_2:=\cb\cap W(x_{k+1},\cdots,x_n)$. Then

\begin{align*}
&\cb_1\supset \sum_{i=1}^k\bbk\partial_i+ \sum_{i=1}^k\bbk x_i\partial_i;\cr
&W(x_{k+1},\cdots,x_n)^{(\ttt_{n-k})}_0\supset \cb_2\supset \sum_{i=k+1}^n\bbk(1+x_i)\partial_i.
\end{align*}
%Note that $\cb_2$ is a solvable.
The further arguments will be divided into two cases: whether or not $\cb_2$ contains nilpotent elements.

{\bf {The first case:}} there is no nilpotent element in $\cb_2$. Then $\cb_2=\bbk(1+x_{k+1})\partial_{k+1}+\cdots+\bbk(1+x_n)\partial_n$. We write
\begin{align}\label{Bfirstcase-1}
\cb=\cb_1+\sum_{(b_1,\cdots,b_k;c_{k+1},\cdots,c_n;i)\in I}\bbk x_1^{b_1}\cdots x_k^{b_k}(1+x_{k+1})^{c_{k+1}}\cdots(1+x_n)^{c_n}\partial_i+\cb_2;
\end{align}
for a subset $I$ of $\Gamma=\bbp^k\times \bbp^{n-k}\times \{1,\cdots,k\}$ where $\bbp:= \{0,1,\cdots,p-1\}$. Denote the above second sum formula in (\ref{Bfirstcase-1}) by $\cb_1'$, surely $\cb_1'\subset \cb$. %Now  we  have
%\begin{align}\label{Bprime}
%$\cb=\cb_1+\cb_1'+\cb_2$.
%\end{align}

 Let us introduce a new subset of $\overline I$ of  $\Gamma$, associated with $I$
 \begin{align*}
 \overline I:=\{(b_1,\cdots,b_k; \gamma_{k+1},\cdots,\gamma_n; i)\mid  &\gamma_i=0,1,\cdots,p-1; \forall i=k+1,\cdots,n \newline
 \cr
 &\mbox{if exists }(b_1,\cdots,b_k;c_{k+1},\cdots,c_n;i)\in I\}.
 \end{align*}
We simply write  elements of $\overline I$ as $(b_1,\cdots,b_k;\infty,\cdots,\infty;i)$, which just means the existence of some $(b_1,\cdots,b_k;-,\cdots,-;i)\in I$.
Set $$\overline{\cb'_1}=\sum_{(b_1,\cdots,b_k;\infty,\cdots,\infty;i)\in \overline I}\bbk x_1^{b_1}\cdots x_k^{b_k}(1+x_{k+1})^\infty\cdots (1+x_n)^\infty\partial_i.$$
And set
$$\overline \cb=\cb_1+\overline{\cb'_1}+\cb_2.$$
We claim  that
 \begin{align*}
 \overline\cb \mbox{ is a solvable subalgebra of }W(n) \mbox{ containing }\cb,
\end{align*}
thereby $\cb$ coincides with $\overline\cb$ by the maximality of $\cb$.
To prove the above claim, we first need to check that the linear subspace $\overline \cb$ is a subalgebra. For this, note that $\cb_1+\overline{\cb'_1}$ is normalized by $\cb_2$. We only need to check $\cb_1+\overline{\cb'_1}$ ($=\overline{\cb'_1}$) is a subalgebra.
 In order  to check this,  we take any two elements  $\overline X, \overline X'\in \overline{\cb'_1}$ with $\overline X=x_1^{b_1}\cdots x_k^{b_k}(1+x_{k+1})^\infty\cdots (1+x_n)^\infty\partial_i$, $\overline X'=x_1^{b'_1}\cdots x_k^{b'_k}(1+x_{k+1})^\infty\cdots (1+x_n)^\infty\partial_j$. Correspondingly we can find two elements  $X$ and $X'$ in $\cb_1'$ by the definition as below
\begin{align*}
X&=x_1^{b_1}\cdots x_k^{b_k}(1+x_{k+1})^{c_{k+1}}\cdots (1+x_n)^{c_n}\partial_i \cr
  X'&= x_1^{b'_1}\cdots x_k^{b'_k}(1+x_{k+1})^{c'_{k+1}}\cdots (1+x_n)^{c'_n}\partial_j
   \end{align*}
   for some $(c_{k+1},\cdots,c_n),(c'_{k+1},\cdots,c'_n)\in\bbp^{n-k}$. And then
\begin{align}\label{Lieproduct}
[\overline X,\overline X']=(1+x_{k+1})^\infty\cdots (1+x_n)^\infty[x_1^{b_1}\cdots x_k^{b_k}\partial_i,x_1^{b'_1}\cdots x_k^{b'_k}\partial_j]\in \overline\cb'_1
\end{align}
because $[X,X']\in\cb'_1$, and $\cb'_1\subset \cb$ is $\ttt_{n-k}$-graded. The solvableness of $\overline\cb$ follows from (\ref{Lieproduct}), along  with the normalization of $\overline\cb'_1$ to $\cb_1$, and the solvableness of $\cb$. Hence,  $\cb=\overline\cb$.
On the other hand, the structure of $\overline\cb$ implies that $\cb+\bbk\partial_n=\overline\cb+\bbk\partial_n$ is a solvable subalgebra of $W(n)$, which contradicts  the maximal solvableness of $\cb$ because $\cb+\bbk\partial_n\supsetneqq \cb$.
So the first  case could not happen.

 {\bf{The second case:}} there exist some nilpotent elements in $\cb_2$. In this case,  $\cb_2 \supsetneqq\ttt_{n-k}(x_{k+1},\cdots, x_n)=\bbk(1+x_{k+1})\partial_{k+1}+\cdots+\bbk(1+x_{n})\partial_n$. Note that $\cb$ is a maximal solvable subalgebra containing $\ttt_n$, and $\cb_2$ is $\ttt_{n-k}(x_{k+1},\cdots,x_n)$-graded.
 We may suppose that there exists a   nilpotent element of "monomial form"  as below
 \begin{align*}%\label{nilp}
 X=(1+x_{k+1})^{a_{k+1}}\cdots (1+x_n)^{a_n}\partial_q
 \end{align*}
 which satisfies $a_q\ne 1$ (called {\sl {the  "q"-condition}}).  At first,  by a direct computation, any derivations of the form
 $(1+x_q)(\prod_{i\ne q} (1+x_{i})^{a_{i}})\partial_q$ are non-nilpotent,  thereby a nilpotent element of "monomial form"  %in $\cb_2$
 must satisfy the "q"-condition. % if such a nilpotent element exists.
  Next, we show that such a nilpotent element of "monomial form" really exists in $\cb_2$. By  $\ttt_{n-k}(x_{k+1},\cdots,x_n)$-grading, we can first take a nilpotent element $X'$ to be a linear combination  of some "monomial form" elements of a   $\ttt_{n-k}(x_{k+1},\cdots,x_n)$-grading $(d_{k+1},\cdots,d_n)$.  Under  a suitable conjugation similar to the forthcoming (\ref{autoauto}) (only involving $x_{k+1},\cdots,x_n$ and satisfying the forthcoming ($\Phi$-2)), we can get a "monomial" summand from the combination such that its $\ttt_{n-k}(x_{k+1},\cdots,x_n)$-grading is different from the other summands. Then  such a "monomial" summand  must be a nilpotent element in $\cb_2$. We take it as a desired $X$ (the above computation is trivial and tedious, we omit the details here).

  %because $\cb_2$ is $\ttt_{n-r}(x_{r+1},\cdots,x_n)$-graded.
  %because of the nilpotency of $X$.% implies that
    By the same way as in \S\ref{Automorphisms}, we  define an automorphism $\varphi$ of the truncated polynomial $A(n)$ via
\begin{align}\label{autoauto}
&\varphi: x_j\mapsto x_j \mbox{ for } j\in\{1,\cdots,n\}\backslash \{q\}, \cr
&\varphi: x_q\mapsto (1+x_q)\prod_{j\in\{k+1,\cdots,n\}, j\neq q}(1+x_j)^{d_j}-1
\end{align}
where $d_j=d_q(p-a_j)\in \{0,1\cdots,p-1\}$ with $d_q(a_q-1)\equiv 1\mod p$. Then $\varphi$ induces an automorphism $\overline\varphi$ of $W(n)$, denoted by $\Phi$.
\iffalse
The inverse of $\varphi$ is easily presented as follows:
\begin{align}
&\varphi^{-1}: x_j\mapsto x_j \mbox{ for } j\in\{1,\cdots,n\}\backslash \{q\}, \cr
&\varphi^{-1}: x_q\mapsto (1+x_q)\prod_{j\in\{r+1,\cdots,n\}, j\neq q}(1+x_j)^{p-d_j}-1.
\end{align}
\fi
Then $\Phi|_{\cb_1}=$ identity. We have
\begin{align*}
\Phi(\partial_q)&=\varphi\circ\partial_q\circ\varphi^{-1}\cr
&=\prod_{j\in\{k+1,\cdots,n\}, j\neq q}(1+x_j)^{p-d_j}\partial_q.
\end{align*}
And then we have
\begin{align*}
\Phi(X)%=&\Phi((1+x_{r+1})^{a_{r+1}}\cdots (1+x_n)^{a_n}\partial_q)\cr
%=&\Phi((\prod_{j\in\{r+1,\cdots,n\}, j\neq q}(1+x_j)^{a_j}) (1+x_q)^{a_q}\partial_q)\cr
%=&\varphi((\prod_{j\in\{r+1,\cdots,n\}, j\neq q}(1+x_j)^{a_j}) (1+x_q)^{a_q})\Phi(\partial_q)\cr
=&(\prod_{j\in\{k+1,\cdots,n\}, j\neq q}(1+x_j)^{a_j})\varphi((1+x_q)^{a_q})\Phi(\partial_q)\cr
%=&(\prod_{j\in\{r+1,\cdots,n\}, j\neq q}(1+x_j)^{a_j})(1+x_q)^{a_q}\prod_{j\in\{r+1,\cdots,n\}, j\neq q}(1+x_j)^{d_ja_q})\cr
%&&\prod_{j\in\{r+1,\cdots,n\}, j\neq q}(1+x_j)^{p-d_j}\partial_q       \cr
=&(1+x_q)^{a_q}\partial_q.
\end{align*}
%\end{align}
Furthermore, $\Phi((1+x_j)\partial_j=(1+x_j)\partial_j$ for $j\in \{ k+1,\cdots,n\}\backslash\{q\}$. And
%\begin{align*}
$$\Phi((1+x_q)\partial_q)=\varphi(1+x_q)\Phi(\partial_q)
%&=(1+ x_q)\prod_{j\in\{r+1,\cdots,n\}, j\neq q}(1+x_j)^{d_j} \prod_{j\in\{r+1,\cdots,n\}, j\neq q}(1+x_j)^{p-d_j}\partial_q\cr
=(1+x_q)\partial_q.$$
%\end{align*}
Hence $\Phi$ satisfies  the following properties
\begin{itemize}
\item[($\Phi$-1)] $\Phi|_{\ttt_{n-k}(x_{k+1},\cdots, x_n)}=$ identity.
\item[($\Phi$-2)] Under $\Phi$, $\Phi(\cb)=\cb_1+\Phi(\cb_2)$ which contains $\ttt_{n}$, and intersects with $W(n)^{(n)}_{[-1]}$ at $\sum_{i=1}^k\bbk\partial_i$.
\item[($\Phi$-3)] It is a specially important consequence that $\Phi(\cb_2)$ admits a nilpotent element $\Phi(X)=(1+x_q)^{a_q}\partial_q$.
\end{itemize}
 Next we define an automorphism $\Psi$ of $W(n)$ via
 $$\psi(x_j)=\begin{cases} x_j, &\mbox{ if } j\neq q;\cr
(1+x_q)^{b_q}-1, &\mbox{ if } j=q;
\end{cases}$$
where $b_q\in\{1,\cdots,p-1\}$ with $(p-a_q+1)b_q\equiv 1\mod p$. By a straightforward  computations as in (\ref{compWitt}), we have
\iffalse
the following items for $\Psi(\Phi(\cb)))$:
\begin{itemize}
\item[($\Psi$-1)] $\Psi\circ\Phi|_{\ttt_{n-k}(x_{k+1},\cdots, x_n)}=$ identity.
\item[($\Psi$-2)]  $\Psi(\Phi(\cb))=\cb_1+\Psi(\Phi(\cb_2))$ which contains $\ttt_{n}$, and intersects with $W(n)^{(\ttt_n)}_{[-1]}$ at $\sum_{i=1}^r\bbk\partial_i+\bbk\partial_q$.
\item[($\Psi$-3)] It is a specially important consequence that $\Psi(\Phi(\cb_2))$ admits a nilpotent element $\partial_q$.
\end{itemize}
 Thus,
 \fi
 $\Psi\circ\Phi(\cb)\supset \ttt_n$ with $\dim\Psi\circ\Phi(\cb)^{(\ttt_n)}_{[-1]}=k'+1$.
 The inductive hypothesis yields that  $\Psi\circ\Phi(\cb)$ is conjugate to $\cb_n$. We complete the proof.
\end{proof}

\subsection{} It is the position for us to investigate the general cases. %Recall the meaning of $\textsf{r}(\cb)$
\begin{lemma} \label{conjugate to Br} Let $\cb$ be any given homogeneous Borel subalgebra of $W(n)$.   Then $\cb$ is conjugate to $\tsb_r$, where $r=\textsf{r}(\cb)$.
\end{lemma}

\begin{proof} If $r=n$, then the statement holds, thanks to Lemma \ref{fullborel}. In the following argument, we assume $r<n$.
Up to conjugation, we might as well assume $\cb$ contains $\ttt_r=\bbk x_1\partial_1+\bbk x_2\partial_2+\cdots +\bbk x_{n-r}\partial_{n-r}+\bbk(1+x_{n-r+1})\partial_{n-r+1}+\cdots+\bbk(1+x_n)\partial_n$. Then $\cb$ is $\ttt_r$-graded, owing to the homogeneous conditon.
%Up to conjugation, we might as well assume the following items:
  According to the assumption, we have the following observation:
\begin{itemize}
\item[(r-0)] $\cb^{(\ttt_r)}_{[-1]}\subset \bbk\partial_{n-r+1}+\cdots+\bbk\partial_n$.
\end{itemize}
Note that $\cb^{(\ttt_r)}_{[0]}\cap W(x_1,\cdots,x_{n-r})$ is a solvable subalgebra of $W(x_1,\cdots,x_{n-r})_{[0]}$ containing the maximal torus $\ttt_0(x_1,\cdots,x_{n-r})$. And $G_0=\GL(n,\bbk)\supset \GL(n-r,\bbk)\times \GL(r)$. By a result on Borel subalgebra conjugacy classes of the classical Lie algebras (cf. \cite[\S11.3-4]{Hum-1}), we can further  make the following assumption without loss of  generality.
\begin{itemize}
\item[(r-1)] $\cb^{(\ttt_r)}_{[0]}\cap W(x_1,\cdots,x_{n-r})\subset \sum_{i\leqq j;\; i,j=1,\cdots,n-r}\bbk x_i\partial_j$.
%\item[(r-2)] $\cb^{(\ttt_r)}_{[0]}\cap W(x_{n-r+1},\cdots,x_n)\subset \sum_{i\leqq j;\; i,j=n-r+1,\cdots,n}\bbk(1+x_i)\partial_j$.
%\item[(r-3)] $\cb^{(\ttt_r)}_{[-1]}=\bbk\partial_{n-s+1}+\cdots+\bbk\partial_{n}$.
\end{itemize}
Set $\cb_1=\cb\cap W(x_1,\cdots,x_{n-r})$ which is solvable . Then $\ttt_0(x_1,\cdots,x_{n-r})\subset \cb_1\subset W(n-r)_0$. According to Lemma \ref{The Third Key Lemma}, we can additionally make the following  assumption without loss of generality:
\begin{itemize}
\item[(r-2)] $\cb_1\subset\mbox{ the Borel subalgebra }\tsb_0(x_1,\cdots,x_{n-r})$.
\end{itemize}
Set $\cb_2=\cb\cap W(x_{n-r+1},\cdots,x_n)$. Then $\cb_2$ is solvable and contains the maximal torus $$\ttt_{r}(x_{n-r+1},\cdots,x_n)$$ of $W(x_{n-r+1},\cdots,x_n)$. Observe that
$$\Aut(W(x_1,\cdots,x_{n-r}))\times \Aut(W(x_{n-r+1},\cdots,x_n))\subset \Aut(W(n)).$$
So
we may make the following assumption  additionally, taking Lemma \ref{fullborel} into account:%, without loss of generality
\begin{itemize}
\item[(r-3)] $\cb_2\subset \Theta(\tsb_{r}(x_{n-r+1},\cdots,x_n))$ for a certain  $\Theta\in\Aut(W(x_{n-r+1},\cdots,x_n))$.% arises from $\theta\in\Aut(A(x_{n-r+1},\cdots,x_n))$ defined via $\theta(x_{n-r+i})=(1+x_{n-r+i})^{p-1}-1$ for $i=1,\cdots,r$.  This does make sense, according to the argument in \S\ref{Automorphisms}.
\end{itemize}
 Set ${\bf{x}}=(x_1,\cdots,x_{n-r})$ and ${\bf{y}}=(y_{n-r+1},\cdots,y_n)$ for $y_i=1+x_i$.
Consider a subspace $Q$ as below (keeping the notations in  Conventions \ref{conventions}):
$$Q=\sum_{i=1}^{n-r}\sum_{{\bf{a}}(i),{\bf{b}}(i)}\bbk{\bf{x}}^{{\bf{a}}(i)}{\bf{y}}^{{\bf{b}}(i)} \partial_i +\sum_{j=n-r+1}^n\sum_{{\bf{a}}(j);{\bf{b}}(j)}\bbk{\bf{x}}^{{\bf{a}}(j)}{\bf{y}}^{{\bf{b}}(j)} \partial_j,
$$
where ${\bf{a}}(i):=(a_1,\cdots,a_{n-r})\in \bbp^{n-r}$ is subjected to the condition that either $|{\bf{a}}(i)|>1$ or $|{\bf{a}}(i)|=1=a_1+\cdots+a_i$, while ${\bf{a}}(j)\in \bbp^{n-r}$ is subjected to the condition $|{\bf{a}}(j)|>0$, and ${\bf{b}}(-):=(b_{n-r+1},\cdots,b_n)$ runs through $\bbp^{r}$ for $(-)=(i), (j)$.
It is not hard to see that $Q$ is stabilized under the action of $\Aut(W(x_{n-r+1},\cdots,x_n))$, and that the following  statements hold
\begin{itemize}
\item[(r-4)] $Q$ is a subalgebra normalizing $\tsb_0(x_1,\cdots,x_{n-r})$. And $Q+\tsb_0(x_1,\cdots,x_{n-r})$ normalizes $\Theta(\tsb_r(x_{n-r+1},$ $\cdots,x_n))$.
\item[(r-5)]  $\Theta (\tsb_r)=\tsb_0(x_1,\cdots,x_{n-r})+ Q+\Theta(\tsb_r(x_{n-r+1},\cdots,x_n))$ (note that $\Theta$ can be naturally regarded as an automorphism of $W(n)$).
\end{itemize}
Set $\cb_Q:=\cb_1+Q+\cb_2$.  Then $\cb_Q\subset \Theta(\tsb_r)$.  From (r-2), (r-3) and (r-4), we know that $\cb_Q$ is a solvable subalgebra. We will finally show the following inclusion relation%, up to conjugation
\begin{itemize}
\item[(r-6)]
$\cb_Q\supset \cb$.
\end{itemize}
In order to prove (r-6), we first claim that  $\cb$ may contain, up to conjugation,
%, up to conjugation,
 neither elements of the forthcoming form (\ref{formY}), nor elements of the forthcoming form (\ref{formZ}), presented as
\begin{align} \label{formY}
 Y={\bf{y}}^{{\bf{b}}(q)}\partial_q\mbox{ with }q\in\{1,2,\cdots,n-r\};
 \end{align}
and
\begin{align}\label{formZ}
Z=x_m{\bf{y}}^{{\bf{b}}(q)}\partial_q\mbox{ with } m,q\in\{1,2,\cdots,n-r\}, m>q.
\end{align}
Suppose there exists in $\cb$ such an element $Y$  of form (\ref{formY}). Clearly, $|{\bf{b}}(q)|>0$ by (r-0).
For $Y=(1+x_{n-r+1})^{b_{n-r+1}}\cdots (1+x_n)^{b_n}\partial_q$.  According to the argument of \S\ref{Automorphisms}, one can consider an automorphism $\Omega$ of $W(n)$ induced by $\omega\in\Aut(A(n))$, which is defined via
\begin{align*}
&\omega: x_i\mapsto x_i \mbox{ for } i\in\{1,\cdots,n\}\backslash \{q\}, \cr
&\omega: x_q\mapsto x_q\prod_{j\in\{n-r+1,\cdots,n\}}(1+x_j)^{b_j}
\end{align*}
 Then $\Omega(Y)=\partial_q$, while $\Omega(x_q\partial_q)=x_q\partial_q$, and we have $\Omega(\cb)\supset \sum_{i=1, i\neq q}^{n-r}\bbk x_i\partial_i+\bbk(1+x_q)\partial_q+\sum_{i=n-r+1}^n \bbk   (1+x_i)\partial_i\cong \ttt_{r+1}$, which contradicts the assumption $\textsf{r}(\cb)=r$.
  Thus we have proved that $\cb$ does not contain any elements of the form (\ref{formY}).

For (\ref{formZ}), we first conclude that
\begin{align}\label{STATE}
\cb &\mbox{ does not contain any pair of elements of the forms:}\cr
&Z(m,q)=x_m(1+x_{n-r+1})^{b_{n-r+1}}\cdots (1+x_n)^{b_n}\partial_q,\cr
&Z(q,m)=x_q (1+x_{n-r+1})^{c_{n-r+1}}\cdots (1+x_n)^{c_n}\partial_m,
\end{align}
with $m,q\in\{1,2,\cdots,n-r\}, m>q$. Otherwise, suppose that $\cb$ contains a pair of elements as above. Then one can easily find a non-solvable subalgebra in $\cb$ arising  from such pairs, which contradicts the solvableness of $\cb$.
\iffalse
Suppose that there are such a pair $Z(m,q)$ and $Z(q,m)$ appearing in $\cb$.
For simplicity of our arguments below, we abbreviate  $Z(m,q)$, $Z(q,m)$ to $\Pi^{\bf{b}}x_m\partial_q$, $\Pi^{\bf{c}}x_q\partial_m$ respectively, where $\Pi^{\bf{b}}$, $\Pi^{\bf{c}}$ denotes respectively the products $(1+x_{n-r+1})^{b_{n-r+1}}\cdots (1+x_n)^{b_n}$, and $(1+x_{n-r+1})^{c_{n-r+1}}\cdots (1+x_n)^{c_n}$. We further denote $H_m$ for $x_m\partial_m$, and  $H_q$ for $x_q\partial_q$.
 Note that  $\cb$ contains the subalgebra generated by $Z(m,q)$ and $Z(q,m)$, denoted by $L(m,q)$. We have the Lie products in $L(m,q)$ as below
  \begin{align}\label{PiSL2}
&[\Pi^{\bf{b}}x_m\partial_q, \Pi^{\bf{c}}x_q\partial_m]=\Pi^{{\bf{b}}+{\bf{c}}}(H_m-H_q);\cr
&[\Pi^{\bf{b+c}}(H_m-H_q),\Pi^{\bf{b}}x_m\partial_q]=\Pi^{2{\bf{b}}+\bf{c}} (2x_m\partial_q);\cr
&[\Pi^{\bf{b+c}}(H_m-H_q), \Pi^{\bf{c}}x_q\partial_m]=\Pi^{{\bf{b}}+2{\bf{c}}}(-2 x_q\partial_m).
\end{align}
Note that all elements like $\Pi^{\bf{b}}$ are invertible in $A(x_{n-r+1},\cdots,x_n)$, thereby being invertible linear transformation of $W(n)$. Thus  $L(m,q)$ is not solvable, which contradicts  the solvableness of $\cb$. We complete the proof of (\ref{STATE}).
 \fi
 Combining the statement (\ref{STATE}) with the assumption (r-1), we may conclude that the maximal solvable  subalgebra $\cb$ does not contain any elements of the form (\ref{formZ})
\iffalse

 Note that if  $\cb$ contains $x_q\partial_m$ as in (r-1), then the statement (\ref{STATE}) already assures that $\cb$ does not contain any elements of the form (\ref{formZ}). Thus, we can assure  that $\cb$ does not contain any elements of the form (\ref{formZ})
 \fi
 (if necessary, we change the order of the indeterminants $x_i$ for $i=1,\cdots,n-r$, which gives rise to some automorphism of $W(n)$).

  In order to  complete the proof of (r-6),  we only need to show that there is no  element in $\cb$, which is a combination of some elements of the form (\ref{formY}) or (\ref{formZ}). Notice that all summands in such a combination admit different $\ttt_r$-gradings (different elements of the form   (\ref{formY}) or (\ref{formZ}) admit different $\ttt_r$-gradings).  On the other hand,  $\cb$ is $\ttt_r$-graded (the homogeneous condition).  If  $\cb$ contains such a combination, then $\cb$ contains its all summands. This contradicts what we just concluded. So $\cb$ does not contain  such a combination. We have completed the proof of (r-6).

It follows from (r-2)-(r-6) that  $\cb\subset \cb_Q\subset\Theta(\tsb_r)$. The maximality of $\cb$ implies  $\cb=\Theta(\tsb_r)$. We complete the proof.
\end{proof}

We have a direct consequence
\begin{corollary} \label{Invar} The following statements hold.
\begin{itemize}
\item[(1)] Any homogeneous Borel subalgebra of $W(n)$ contains a maximal torus conjugate to $\ttt_0$.

(Therefore, for given  homogeneous Borel subalgebra $\cb$ of $W(n)$, we can take its conjugacy $\cb^0$ such that  $\cb^0\supset\ttt_0$.)

\item[(2)] If set $\textsf{d}(\cb)=\dim \cb^0_{[-1]}$, then
$\textsf{d}(\cb)=\textsf{r}(\cb)$.
\end{itemize}
\end{corollary}

So, there are two equal invariants of Borel subalgebras: $\textsf{d}(\cb)$ and $\textsf{r}(\cb)$ under $\Aut(W(n))$.% Thus, the isomorphism classes of Borel subalgebras are parameterized by  $\tsb_i$, $i=0,1,\cdots, n$.

\subsection{}%{Main theorem}
By Lemma \ref{conjugate to Br}, Corollary \ref{Invar} and \cite[Proposition 3.3]{BFS} (note that in \cite{BFS}, the order of the parameters in the subscripts of the notations $\{\ttt_r\}$ is inverted to the one in the present paper),
% along with Remark \ref{cartantroi},
    we have
\begin{theorem} \label{maintheorem} Assume that  $\ggg=W(n)$, $G=\Aut(W(n))$, and that the characteristic $p$ of the ground field $\bbk$ is bigger than $3$. Then the following statements hold.
\begin{itemize}
\item[(1)] All standard Borel subalgebras $\tsb_i$, $i=0,1,\cdots,n$ are homogeneous.
\item[(2)] There are $(n+1)$ conjugacy classes of homogeneous Borel subalgebras of $W(n)$ under $G$.
The standard Borel subalgebras $\tsb_i$, $i=0,1,\cdots, n$ are representatives of $(n+1)$ conjugacy classes.
\item[(3)] There are only one conjugacy class of generic homogeneous Borel subalgebras, which is conjugate to $\tsb_n$.
\end{itemize}
\end{theorem}

\begin{remark}
\iffalse
 \label{remarkongeneric} (1) Let $\ggg=W(n)$ and $G$ be the adjoint group of $\ggg$. By Remark \ref{nilpotentcone}, we have $\cn(\ggg)=\overline{G\cdot \cn(\tsb_n)}$.  This result has been obtained in \cite{Pr3}, there Premet gave a stronger result that the nilpotent cone coincides with the closure of the orbit $G\cdot \frak D$ of a regular nilpotent element $\mathfrak D$. With generic Borel subalgebras, we can read something more from this viewpoint. That $\frak D$ falls in the generic Borel subalgebra $\tsb_n$, not falling in any  non-generic Borel subalgebras.

(2) When taking non-generic Borel subalgebras of $\ggg=W(n)$, we can see  the assertion in Remark \ref{nilpotentcone} may not  hold ever. For example, take $\cb=B_0$, then the $G$-saturation $G\cdot\cb\subset \ggg_0$,  the Zariski closure of $G\cdot\cb$ is still contained in $\ggg_0$, not equal to $\ggg$. It is not the case of the first assertion of the proposition. Furthermore, by \cite[Lemma 4 and Lemma 5]{Pr3}, $\dim G\cdot \cn(\cb)<np^n-n=\dim\cn(\ggg)$. So $\overline {G\cdot\cn(\cb)}\subsetneqq \cn(\ggg)$, which is not the case of the second assertion  of the proposition.

(3) The dimensions of all Borel subalgebras of $W(n)$ are presented as below, via $\tsb_r, r=0,1,\cdots,n$:
\begin{align*}
\dim \tsb_r=&\dim\tsb_{0}(x_1,\cdots,x_{n-r})+\dim Q_r +\dim\tsb_{r}(x_{n-r+1},\cdots,x_n)\cr
=&((n-r)p^{n-r}-{(n-r)(n-r-1)\over 2})+((n-r)(p^{n-r}-{n-r-1\over 2})(p^r-1)+\cr
&r(p^{n-r}-1)p^r)+ {2(p^r-1)\over p-1}\cr
=&np^n-{(n-r)(n-r-1)\over 2}p^r+{2(p^r-1)\over p-1}-rp^r.
\end{align*}

(4)
\fi
In view of  \cite[Theorem D]{HS},  it is reasonable to expect
% the converse statement of the above corollary, this is to say,
that all maximal solvable subalgebras of $W(n)$ are Borel subalgebras  if and only if $p>n$. When $p\leqslant n$, the  situation of conjugacy of maximal solvable subalgebras would become very complicated.
\iffalse
(5) By a recent result of Premet (cf. \cite{Pr4}),  we can give a description of generic elements in $\ggg=W(n)$. For any $X\in \ggg$ with the Jordan-Chevalley-Seligman decomposition $X=X_s+X_n$, $X$ is generic if and only if the following holds: there is an $g\in G$ such that
$$ g\cdot C_\ggg(X_s)=\ttt_n.$$
\fi
\end{remark}

\section{Solvable subgroups associated with standard Borel subalgebras}

Let $(\ggg,[p])$ be a restricted Lie algebra, and $G$ the adjoint group of $\ggg$. %For a Borel subalgebra $\mathcal B$ we call $B:=\Stab_G(\mathcalB)$  the subgroup associated with Borel subalgebras of $\ggg$.
We are  specially interested in  $\Stab_G(\cb_\gen)$ and $G\slash\Stab_G(\cb_\gen)$ for a generic Borel subalgebra $\cb_\gen$ when $\ggg$ admits generic tori. In this section, we investigate the subgroups $\Stab_G(\cb_r)$  associated with standard Borel algebras in the case $\ggg=W(n)$. Keep the notations  in Theorem \ref{wilsonauto}.
\subsection{} For the standard homogeneous Borel subalgebras $\tsb_r$ of $W(n)$, denote $S_r=\Stab_G(\tsb_r)$, $r=0,1,\cdots,n$. Keeping it in mind that  $G_0$ is isomorphic to $\GL(n,\bbk)$ (cf. Theorem \ref{wilsonauto}), we have a standard Borel subgroup $B_0$ of $G_0$ which corresponds to the one consisting of  invertible upper-triangular  matrices of $\GL(n,\bbk)$.
Then we have the following general description of $S_r$.
\begin{prop} Let $B_0$ be the standard Borel subgroup of $G_0$ corresponding to the one consisting of invertible upper-triangular matrices of $\GL(n,\bbk)$. The following statements hold.
\begin{itemize}
\item[(1)] The solvable subgroup  $S_0$ coincides with the Borel subgroup $B_0\ltimes U$ of $G$.
\item[(2)] For $r=0, 1,\cdots,n$, $S_r=B_0\ltimes U_r$ is a connected subgroup of the Borel subgroup $B_0\ltimes U$, where $U_r=U\cap\Stab_G(\tsb_r)$.
\end{itemize}
\end{prop}

\begin{proof} The first assertion  is immediate. For (2), it is easily verified  that $S_r=B_0\ltimes U_r$. It remains to prove the connectedness of $U_r$. We identify $B_0$ with the standard Borel subgroup of $\GL(n,\bbk)$ consisting of invertible upper-triangular matrices, which has a maximal torus group $T_0$ consisting of invertible diagonal matrices. Obviously, $T_0$ normalizes $U_r$.  Take $\tau(t)=\diag(t,t,\cdots,t)\in T_0$ for a given $t\in \bbk\backslash\{0\}$. Given a unipotent element $g\in U_r$, all $\tau(t)g\tau(t)^{-1}$ belong to the same connected component of $U_r$ as $g$ lies in. Note that $G$ can be identified with $\Aut(A(n))$. Under this identification, we can easily see that $\lim_{t\rightarrow 0}\tau(t)g\tau(t)^{-1}=\id_{A(n)}$. Hence $U_r$ is connected. We complete the proof.
\end{proof}

\subsection{}Set $\ggg=W(n)$ and $G=\Aut(W(n))$ described as in Theorem \ref{wilsonauto}. We use the notations $B_\gen$ for $\Stab_G(\tsb_n)$ instead of $S_n$.

\begin{prop}
In aid of identification between $G$ and $\Aut(A(n))$, $B_{\gen}$ can be regarded as a subgroup of the isotropy group of the following algebra flag
\begin{align*}%\label{flag}
\bbk[x_1,\cdots,x_n]\supset \bbk[x_1,\cdots,x_{n-1}]\supset\cdots\supset \bbk[x_1].
\end{align*}
Furthermore, $B_{\gen}=B_0\ltimes U_n$, with
$$U_n=\{g\in U\mid (g-\id)(x_i)\in (\bbk+\bbk x_i)\bbk[x_1,\cdots,x_{i-1}], i=1,\cdots,n\}.$$% Hence $B_\gen$ is a  solvable subgroup of $G$.
\end{prop}

\begin{proof} It is clear that $B_\gen$ contains both $B_0$ and $U_n$, thereby contains $B\ltimes U_n$. Next we prove the inverse inclusion.

For any given $\sigma\in B_\gen$, by Theorem \ref{wilsonauto} we can write $\sigma=\sigma_0\circ \sigma_1$ with $\sigma_0\in G_0$ and $\sigma_1\in U$.
So $\sigma_0$ stabilizes $\bbb\subset \tsb_n$,  which means $\sigma_0\in B_0\subset B_\gen$.
Thus $\sigma_1\in B_\gen$.  It remains to show $\sigma_1\in U_n$. Suppose $\sigma_1$ is not in $U_n$. Then there exists some $i$ such that $\sigma_1(x_i)=x_i+ x_jf_1+x_i^2f_2+f_3$ with $j>i$,  $f_1\in \bbk[x_1,\cdots,x_n], f_2\in \bbk[x_1,\cdots,\widehat{x_j}, \cdots,x_n], f_3\in (\bbk+\bbk x_i)\bbk[x_1,\cdots, \widehat{x_i},\cdots,\widehat{x_j},\cdots,x_n]$, and $f_1$, $f_2$ are not all zero.
  % admitting the degree not less than $1$.
  Here the notations $\widehat{x_i}, \widehat{x_j}$ mean omitting the indeterminants $x_i$, $x_j$ respectively.
Note that $\sigma_1\in U$. By a straightforward computation,  $\sigma_1^{-1}(x_i)=x_i+x_jg_1+x_i^{2}g_2+g_3$, where $g_1, g_2, g_3\in A(n)$ are subjected to the same conditions as the ones for $f_1$, $f_2$ and $f_3$.
Then $\overline\sigma_1(x_i\partial_i)=x_i\partial_i+(x_jf_1+x_i^2f_2+f_3)\partial_i+*$ does not fall in $\tsb_n$. This is a contradiction.
Hence $\sigma$ falls in $B_0\ltimes U_n$. We complete the proof.
\end{proof}

\section*{Additional Notation List}\label{list}

For the truncated polynomial $A(n)$ and its derivation algebra $W(n)$.
\begin{itemize}
\item $\bbp=\{0,1\cdots,p-1\}$.

\item ${\bf{\epsilon}}_i=(\delta_{i,1},\cdots,\delta_{i,n})\in \bbp^n$, $\delta_{i,j}=1$ if $i=j$, and $0$ otherwise.

\item $A(n)=\bbk[T_1,\cdots,T_n]\slash (T_1^p,\cdots, T_n^p)$.

\item $x_i=T_i+(T_1^p,\cdots,T_n^p)\in A(n)$, $y_i=x_i+1$,  and $z_i=x_i$ or $y_i$ according to the prescribed, $i=1,\cdots,n$.

\item $\overline \varphi$ is an automorphism in $W(n)$, induced from $\varphi \in \Aut(A(n))$.

\item ${\bf{x}}^{\bf{a}}=x_1^{a_1}\cdots x_n^{a_n}$; ${\bf{y}}^{\bf{a}}=y_1^{a_1}\cdots y_n^{a_n}$;  and ${\bf{z}}^{\bf{a}}=z_1^{a_1}\cdots z_n^{a_n}$ all in $A(n)$ for  ${\bf{a}}=(a_1,\cdots,a_n)\in \bbp^n$.  More verified forms of them can be found in Convention \ref{conventions}.

\end{itemize}

For maximal tori and related gradations:

\begin{itemize}
\item $\ttt_r=\sum_{i=1}^{n-r} \bbk x_i\partial_i + \sum_{i=(n-r)+1}^n (1+x_i)\partial_i$, $r= 1,\cdots,n$ (see \S\ref{2.3}).

\item $\frakh^{(\ttt_r)}_{[i]}$: $\bbz[\ttt_r]$-graded subspace, here and below $\frakh$ is a subalgebra of $W(n)$ containing $\ttt_r$.% (see \S\ref{2.5}).
\item $\frakh^{[\ttt_r]}_\alpha$:  $\ttt_r$-graded  subspace (root subspace).%  (see \S\ref{2.5}).
\item   $\frakh_{[i]}$:  standard graded subspace (associated with $\ttt_0$), here $\frakh$  contains $\ttt_0$.

\end{itemize}

For  Borel subalgebras:

\begin{itemize}
\item $\bbb$: the standard Borel subalgebra in $W_{[0]}\cong \gl(n)$ (see \S\ref{2.6}).
\item  $L^{(i)}$ for a Lie algebra $L$:  the $i$-th derived ideal (see \S\ref{2.6}).
\item $\tsb_q$:  Borel subspaces of $W(n)$, $q=0,1,\cdots,n$ (see \S\ref{Borels}). All of them are finally proved to be Borel subalgebras, and to parameterize the iso-classes of homogeneous Borel subalgebras of $W(n)$.
 \item $\tsb_0$ and $\tsb_n$:  the nought-switched Borel subspace, and the full-switched Borel subspace of $W(n)$ respectively (see \S\ref{Borels}).
\item $\tsb_0(u_1,\cdots,u_q)$ and $\tsb_q(u_1,\cdots,u_q)$: the same meaning as above with respect to $u_1,\cdots,u_q$ (see Convention \ref{conventions}).
\item  $\textsf{r}(\cb)$: an invariant of homogeneous Borel conjugacies (see \S\ref{r B}).
\end{itemize}

\subsection*{Acknowledgements}  The author  thanks Hao Chang, Rolf Farnsteiner, Zongzhu Lin, Ke Ou, Xin Wen, Husileng Xiao, Yufeng Yao and Mengmeng Zhang for helpful discussions.


\begin{thebibliography} {ABCD}
\bibitem{BMR} R. Bezrukavnikov, I. Mirkovi$\acute{c}$ and D. Rumynin, {\em Localization of modules for a semisimple Lie algebra in prime characteristic}, Annals of Math. 167 (2008), 945-991.


\bibitem{BFS} J.-M. Bois, R. Farnsteiner and B. Shu, {\em Weyl groups for non-classical restricted Lie algebras and the Chevalley  restriction theorem}, Forum Math. 26 (2014), 1333-1379.

\bibitem {Bor} A. Borel, {\em Linear algebraic groups}, second edition, GTM 126,
Springer-Verlag, 1991.

%\bibitem{CS} H. Chang and B. Shu, {\em Generic Borel subalgebras and Weyl groups for restricted Lie algebras}, in preparation.

\bibitem{De1} S. P. Demu\u{s}kin, {\em Cartan subalgebras of the simple Lie $p$-algebras $W_n$ and $S_n$}, Siberian Math. J. 11 (1970), 233-245

\bibitem{De2} S. P. Demu\u{s}kin, {\em Cartan subalgebras of simple nonclassical Lie $p$-algebras}, Math. USSR Izv. 6 (1972), 905-924.

\bibitem{Fa} R. Farnsteiner, {\em Varieties of tori and Cartan subalgebras of restricted Lie algebras}, Trans. Amer. Math. Soc. 356 (2004), 4181-4236.

%\bibitem{FV} R. Farnsteiner and D. Voight, {\em Schemes of tori and the structure of tame restricted Lie algebras}, J. London. Math. Soc. 63 (2001), 553-570.

\bibitem{Ha} R. Hartshorne, {\em Algebraic geometry}, Springer-Verlag, New York, 1977.

\bibitem{HS} S. Herpel and D. I. Stewart, {\em On the smoothness of normalisers and the subalgebra structure of  modular Lie algebras}, arXiv:1402.6280v1 [math.GR].

\bibitem{Hum-1} J. E. Humphreys, {\em Algebraic groups and modular Lie algebras}, Memoirs  A. M. S. (71), Amer. Math. Soc., Providence, PI, 1967.

\bibitem{Hum} J. E. Humphreys, {\em  Linear algebraic groups}, Springer-Verlag, New York, 1981.

\bibitem{Jan3} J. C. Jantzen, {\em Representations of algebraic groups}, second edition, Amer. Math. Soc.,  Providence, PI,  2003.

%\bibitem {Jan} J. C. Jantzen, {\em Representations of Lie algebras in positive characteristic}, Advanced Studies in Pure
%Mathematics 40, 2004, Representation theory of Algebraic Groups and Quantum Groups, pp. 155-218. %, Dordrecht etc. Kluwer, 1998.

\bibitem{Jan2} J. C. Jantzen, {\em Nilpotent orbits in representation theory}, in "Lie Theory"  PM 228, Birkh$\ddot{a}$user Boston, 2004, pp. 1-206.
%\bibitem {MR} I. Mirkovi\'{c}, D. Rumynin, {\em Centers of reducecd
%enveloping algebra}, Math. Z. 231 (1999), 123-132.

%\bibitem{OSX} K. Ou, B. Shu and H. Xiao, {\em Conjugacy of Borel subalgebras of restricted  Lie algebras and the associated solvable algebraic groups (II)}, in preparation.

%\bibitem {MFK} D. Mumford, J. Fogarty and F. Kirwan, {\em Geometric invariant
%theory}, third enlarged edition, Springer-Verlag, 1992.

\bibitem{O} K. Ou, {\em Weyl groups and Geometric setting of Lie alegbras of Cartan type}, Ph. D thesis, East China Normal University, 2016.

\bibitem{Pr1} A. A. Premet, {\em On Cartan subalgebras of Lie $p$-algebras}, Math. USSR Izv. 29 (1987), 145-157.

\bibitem{Pr2} A. A. Premet, {\em Regular Cartan subalgebras and nilpotent elements in restricted Lie algebras}, Math. USSR Sb. 66 (1970), 555-570.

\bibitem{Pr3} A. A. Premet, {\em The theorem on restriction of invariants and nilpotent elements in $W_n$}, Math. USSR Sb. 73 (1992), 135-159.

%\bibitem{Pr4} A. A. Premet, {\em Description of regular derivations of truncated polynomial rings}, arXiv 1405-2426 [MathRA].
\bibitem{PS} A. A. Premet and H. Strade, {\em Classification of finite dimensional simple Lie algebras in prime characteristics},  Representations of algebraic groups, quantum groups, and Lie algebras,  185-214,
Contemp. Math., 413, Amer. Math. Soc., Providence, RI, 2006.


\bibitem{Seligman} G. B. Seligman, {\em Modular Lie algebras}, Springer-Verlag Berlin Heidelberg,  1967.


%\bibitem{SX} B. Shu and H. Xiao, {\em  A note on the description of generic and regular derivations of trucated polynomials}, in preparation.

\bibitem {Springer} T. A. Springer, {\em Linear algebraic groups}, Sencond
edition, Birkh\"{a}user, 1998.

\bibitem{St} H. Strade, {\em Simple Lie algebras over fields of positive charactersitic I. Structure theory}, Walter de Gruyter, Berlin, 2004.

\bibitem{SF} H. Strade and R. Farnsteiner, {\em Modular Lie algebras and their representations},  Marcel Dekker, New York, 1988.


%\bibitem{WCL} J. Wei, H. Chang and X. Lu, {\em The variety of nilpotent elements and invariant polynomial funcitons on the special algebra $S_n$}, Forum. Math. DOI 10.1515/forum-2012-0163.

%\bibitem{W} J. Wei, {\em The nilpotent variety and invariant polynomial functions in the Hamiltonian algebra}, arXiv:1401.6532.
\bibitem{Wilson} R. L. Wilson, {\em Automorphisms of graded Lie algebras of Cartan type}, Comm. Algebra 3 (1975), 591-613.
\bibitem{YC} Y. Yao and H. Chang, {\em Borel subalgebras of the Witt algebra $W_1$},  arXiv:1212.4349 [MathRT].
\end{thebibliography}
\end{document}